\documentclass[11pt,reqno]{amsart}
\usepackage[margin=1in]{geometry}
\usepackage{algorithm,algpseudocode}
\usepackage[fleqn,tbtags]{mathtools}
\usepackage{pifont}
\usepackage[shortlabels]{enumitem}
\usepackage{verbatim}
\usepackage{tikz}
\usepackage{filecontents}
\usepackage{placeins}
\usepackage{mlmodern}
\usepackage{extarrows}
\usepackage{graphicx}
\usepackage{calc}
\usepackage{adjustbox}
\usepackage{amssymb}
\usepackage{tabu}
\usepackage{cases}
\usepackage{subcaption}
\usepackage{pifont} 

\setlist[description]{font=\normalfont\itshape,labelindent=2em}

\usepackage{hyperref,xcolor}
\hypersetup{
    colorlinks,
    linkcolor={red!50!black},
    citecolor={blue!50!black},
    urlcolor={blue!80!black}
}

\let\oh=\circ
\newcommand{\ccirc}{\mathbin{\mathchoice
  {\xcirc\scriptstyle}
  {\xcirc\scriptstyle}
  {\xcirc\scriptscriptstyle}
  {\xcirc\scriptscriptstyle}
}}
\newcommand{\xcirc}[1]{\vcenter{\hbox{$#1\oh$}}}
\let\circ\ccirc

\newcommand{\tp}{{\scriptscriptstyle\mathsf{T}}}

\newcommand{\st}{\sigma}
\newcommand{\F}{{\scriptscriptstyle\mathsf{F}}}
\newcommand{\Par}{{\scriptscriptstyle\mathsf{P}}}

\DeclareMathOperator{\rank}{rank}
\DeclareMathOperator{\tr}{tr}

\DeclareMathOperator{\diag}{diag}

\DeclareMathOperator*{\argmin}{argmin}
\DeclareMathOperator{\conv}{conv}
\DeclareMathOperator{\vecb}{vecb}
\DeclareMathOperator{\vect}{vec}

\theoremstyle{definition}
\newtheorem{theorem}{Theorem}[section]

\newtheorem{lemma}[theorem]{Lemma}
\newtheorem{corollary}[theorem]{Corollary}
\newtheorem{proposition}[theorem]{Proposition}
\newtheorem{example}[theorem]{Example}

\numberwithin{equation}{section}

\newcommand{\C}{\mathbb{C}}
\newcommand{\R}{\mathbb{R}}

\begin{document}
\title{Generalized matrix nearness problems II}
\author{Rongbiao Thomas Wang}
\address{Computational and Applied Mathematics, University of Chicago, Chicago, IL 60637}
\email{rbwang@uchicago.edu}

\author{Chi-Kwong Li}
\address{Department of Mathematics, College of William and Mary, Williamsburg, Virginia 23187}
\email{ckli@math.wm.edu}

\author{Lek-Heng Lim}
\address{Computational and Applied Mathematics, University of Chicago, Chicago, IL 60637}
\email{lekheng@uchicago.edu}

\begin{abstract}
Given a matrix $A$, a matrix nearness problem seeks an $X$ that most closely approximates $A$ in the sense of minimizing  $\lVert A - X\rVert$ under a variety of constraints on $X$. A generalized matrix nearness problem seeks the same but with three given matrices $A,B,C$ and $\lVert A - BXC\rVert$ in place of $\lVert A - X\rVert$. We extend previous studies of the latter problem in three directions: incorporating an affine term, replacing matrix product by Kronecker product in various manners, and generalizing Frobenius norm to any orthogonally invariant norm. We will solve several of these in closed form. For the rest, we develop an iterative algorithm that works for any Schatten norm, proving that it converges to a global minimizer regardless of the initial point. In addition, the algorithm relies purely on numerical linear algebra, and notably does not compute any explicit gradients or subgradients. Along the way, we will also show that there is no Mirsky-type theorem for rank constrained generalized matrix nearness problems.
\end{abstract}

\maketitle

\section{Introduction}\label{sec:intro}

This article is intended to be a sequel to \cite{li22}, in which the authors addressed the \emph{generalized matrix nearness problem}:
\begin{equation}\label{eq:near}
    \min_{X \in \mathcal{S}} \; \lVert A- BXC\rVert_\F ,
\end{equation}
for given $A \in \R^{m \times n}$, $B \in \R^{m \times p}$, $C \in \R^{q \times n}$, and $\lVert \, \cdot \, \rVert_\F $ the Frobenius norm. The problem was solved in \cite{li22} for a variety of common feasible sets $\mathcal{S}$, defined by:
\begin{dingautolist}{192}
    \item rank constraint: $\rank(X) \le r$ for a given integer $r>0$; \label{it:rank}
    \item norm constraint: $\lVert X \rVert_\F \le \rho$ for a given $\rho > 0$; \label{it:norm}
    \item product constraint: $FXG=H$ for given matrices $F,G,H$; \label{it:prod}
    \item spectral constraints: $X$ has a prescribed eigenvalue or eigenvector; \label{it:eig}
    \item symmetry constraints: $X$ is symmetric or skew-symmetric; \label{it:sym}
    \item structure constraints: $X$ is Toeplitz, Hankel, or circulant; \label{it:spec}
    \item positivity constraints: $X$ is positive semidefinite, correlation, nonnegative, stochastic, or doubly stochastic. \label{it:pos}
\end{dingautolist}
The original motivation was \ref{it:rank}, which first appeared in \cite{Sondermann}, was rediscovered in \cite{Fri}, and finally given an essentially one-line solution in \cite{li22}.

While one may simply treat \eqref{eq:near} as a constrained optimization problem and seek a solution by applying general-purpose optimization algorithms, the methods in \cite{li22} are notable in that they solve an ostensible optimization problem \eqref{eq:near} without computing a single gradient. Solutions are either given as closed-form expressions or as the results of a specialized iterative algorithm that is ``zeroth order'' in optimization parlance, i.e., no derivatives involved. This algorithm \cite[Algorithm~5]{li22} is provably convergent (i.e., yields the global minimizer  $X_*$), globally convergent (i.e., regardless of the choice of initial point $X_0$), and linearly convergent (i.e., error $\lVert X_k - X_* \rVert_\F$ decreases exponentially from one step to the next).

In this article, we extend the problem \eqref{eq:near} in three directions:
\begin{description}
\item[affine] adding an affine term $bx^\tp C$ or $Bx c^\tp$ where $B,C$ are matrices and $b,c$ are vectors;

\item[Kronecker] replacing matrix product by Kronecker product $\otimes$, for example, $BXC$ by $B \otimes X \otimes C$, matrix rank by Kronecker rank, eigenpairs $Xw = \lambda w$ by partial trace $\tr_\Par(X(W \otimes I)) = \lambda W$;

\item[Schatten] generalizing Frobenius norm $\lVert \, \cdot \, \rVert_\F$ to an arbitrary Schatten norm.
\end{description}
The solution methods in this article are of the same nature as those in \cite{li22}  --- either a closed-form solution or an iterative solution from a provably and globally convergent algorithm. Regretfully, we are unable to maintain linear convergence rate for these more complex variants of \eqref{eq:near}; our proposed iterative algorithm could only be guaranteed to have sublinear convergence in most cases.

The best Kronecker rank-$r$ approximation problem deserves some historical commentary. Let $m = m_1 m_2$, $n = n_1n_2$. The Kronecker rank of $X \in  \mathbb{R}^{m \times n}$ is defined by
\[
\rank_\otimes(X) = \min \{ r : X = Y_1 \otimes Z_1 + \dots + Y_r \otimes Z_r\},
\]
where the minimum is taken over $Y_i \in \R^{m_1 \times n_1}$, $Z_i \in \R^{m_2 \times n_2}$, $i =1,\dots,r$. This reduces to standard matrix rank when $m_1 = m$, $n_2 = n$ (thus $n_1 = 1 = m_2$); and so the best Kronecker rank-$r$ approximation problem $\min_{\rank_\otimes(X) \le r} \lVert A- X\rVert_\F $ contains the usual  best rank-$r$ approximation problem as a special case. What is a bit surprising is that the Eckart--Young solution of the latter \cite{Eck} via singular value decomposition extends nearly verbatim to the former. This fact has been rediscovered several times, for example, in \cite{allen,kon, vanLoan} --- the authors did not appear to be aware of earlier efforts. We will extend their works to \eqref{eq:near}, providing closed form solutions when $BXC$ is replaced by $B \otimes X \otimes C$ or $(B_1 \otimes B_2)X(C_1 \otimes C_2)$; beyond these cases, our iterative algorithm would yield the required result to machine precision.

We will elaborate on these and other new contributions below.

\subsection{Closed-form solutions}\label{sec:closed}

New closed-form solutions for any matrix approximation problems are hard to come by, given that the topic has been thoroughly explored over the past ninety years since the appearance of \cite{Eck}. Nevertheless we found four in the course of this work. In Section~\ref{sec:closed-form}, we will present these four variants of \eqref{eq:near} that admit closed-form solutions.
\begin{enumerate}[\normalfont(i)]
\item Affine generalized matrix nearness:\label{it:aff}
\[
\min_{X \in \mathcal{S},\, x\in \R^n} \lVert A- B X C + dx^\tp E \rVert_\F \qquad \text{or} \qquad  \min_{X \in \mathcal{S},\, x\in \R^m} \lVert A- B X C + D xe^\tp \rVert_\F,
\]
for any $A \in \R^{m \times n}$, $B \in \R^{m \times p}$, $C \in \R^{p \times n}$, nonsingular $D \in \R^{m \times m}$, nonsingular $E \in \R^{n \times n}$,  $d \in \R^m$, $e \in \R^n$, and any $\mathcal{S} \subseteq \R^{p \times p}$ for which \eqref{eq:near} has a closed-form solution.

\item Separable Kronecker product constraint:\label{it:sep}
\[
\min_{X \in \mathcal{S}} \; \lVert A- B \otimes X \otimes C \rVert_\F,
\]
for any $A \in \R^{m_1 p m_2 \times n_1 q n_2}$, $B \in \R^{m_1 \times n_1}$, $C \in \R^{m_2 \times  n_2}$, and any $\mathcal{S} \subseteq \R^{p \times q}$  for which \eqref{eq:near} has a closed-form solution.

\item Generalized Kronecker rank constraint:\label{it:kron}
\[
\min_{\rank_\otimes(X) \le r} \lVert  A- (B_1 \otimes B_2)X(C_1 \otimes C_2) \rVert_\F,
\]
for any $A \in \R^{m_1m_2 \times n_1n_2}$, $B_1 \in \R^{m_1 \times p_1}$, $B_2 \in \R^{m_2 \times p_2}$, $C_1 \in \R^{q_1 \times n_1}$, $C_2 \in \R^{q_2 \times n_2}$.

\item Prescribed partial trace Kronecker product constraint:\label{it:ptr}
\[
\min_{ \tr_\Par(X(W \otimes I)) = \lambda W } \lVert  A- (B_1 \otimes B_2)X(C_1 \otimes C_2)  \rVert_\F,
\]
for any $A \in \R^{m_1m_2\times n_1n_2}$, $B_1 \in \R^{m_1 \times p}$, $B_2 \in \R^{m_2 \times p}$, $C_1 \in \R^{p \times n_1}$, $C_2 \in \R^{p \times n_2}$, and $\lambda \in \R$. This an analog of the prescribed eigenvalue problem in \ref{it:eig}, where $W$ plays the role of a $\lambda$-eigenvector. The definition of partial trace will be given in Section~\ref{sec:partial-trace}.
\end{enumerate}

\subsection{Nonexistence of Mirsky-type result}\label{sec:imposs}

The closed-form solutions in \cite{Sondermann, Fri, li22} to \ref{it:rank} require the Frobenius norm. One might wonder if they extend to other orthogonally invariant norms $\lVert \, \cdot \, \rVert$, i.e.,  $\lVert UXV \rVert = \lVert X \rVert$ whenever $U$ and $V$ are orthogonal matrices. The underlying motivation is Mirsky's celebrated result \cite{mir}, i.e., the Eckart--Young solution \cite{Eck} for the  best rank-$r$ approximation problem $\min_{\rank(X)\le r} \lVert A- X\rVert$  holds universally across all orthogonally invariant norms $\lVert \, \cdot \, \rVert$.  We show that no Mirsky-type result could possibly hold for \eqref{eq:near} when $B$ and $C$ are not identity matrices.  In fact, we will show in Section~\ref{sec:norm} that the rank constrained problem
\[
    \min_{\rank(X) \le r} \lVert A- BXC\rVert
\]
has different minimizers for different orthogonally invariant norms.

\subsection{Iterative algorithm}\label{sec:algo}

For orthogonally invariant norms other than the Frobenius norm, we neither have nor expect closed form solutions for any of the aforementioned generalized matrix nearness problems; so we will consider iterative solutions. We have to exclude \ref{it:sep} in Section~\ref{sec:closed} below as it involves an objective function of a fundamentally different nature and requires separate treatment.

Let $\mathcal{S}$ be a feasible set defined by any of the constraints \ref{it:rank}--\ref{it:pos} on p.~\pageref{it:rank}; or \ref{it:aff}, \ref{it:kron}, or \ref{it:ptr} in Section~\ref{sec:closed}; or a mixture of these. We will propose an iterative algorithm (Algorithm~\ref{alg:general}) for solving
\begin{equation}\label{eq:f}
\min_{X \in \mathcal{S}} \; f(A-BXC)
\end{equation}
under mild assumptions on the loss function  $f: \R^{m\times n} \to \R$. In principle, it makes sense to consider $f$ of increasing specificity:
\begin{enumerate}[\upshape (a)]
\item $f$ is a subdifferentiable function that attains a minimum on $\{A - BXC : X \in \mathcal{S}\}$; \label{it:sub}
\item $f$ is a convex function; \label{it:cvx}
\item $f$ is an orthogonally invariant convex function; \label{it:inv}
\item $f$ is an orthogonally invariant norm  $\lVert \, \cdot \, \rVert$; \label{it:nor}
\item $f$ is a Schatten norm $\lVert \, \cdot \, \rVert_{\st,p}$ for some $p \in [1,\infty]$; \label{it:sch}
\item $f$ is the Frobenius norm  $\lVert \, \cdot \, \rVert_\F = \lVert \, \cdot \, \rVert_{\st,2}$. \label{it:fro}
\end{enumerate}
Indeed, our convergence result --- that the iterates of Algorithm~\ref{alg:general} converge from any initial point to a global minimizer of \eqref{eq:f} at a sublinear rate --- applies to the most general case \ref{it:sub} and thus to all other cases. 

Nevertheless, the imposition of progressively stricter assumptions \ref{it:cvx}, \ref{it:inv}, \ref{it:nor} does not allow us to say more until we get to \ref{it:sch}, i.e., when $f$ is a Schatten norm as defined in \eqref{eq:sch}. In this case, we are able to obtain explicit closed-form solutions to the two main update steps in Algorithm~\ref{alg:general}, simplifying it to Algorithm~\ref{alg:Schatten}.  We stress that even though we do not have separate results and algorithms for cases \ref{it:cvx}, \ref{it:inv}, \ref{it:nor}, these play a role in our derivation of Algorithm~\ref{alg:Schatten} from Algorithm~\ref{alg:general}.

In addition, the base case \ref{it:fro}, which is just \eqref{eq:near}, is important for a different reason: the first update step in Algorithm~\ref{alg:general} essentially reduces the problem \eqref{eq:f} to \eqref{eq:near}. So Algorithm~\ref{alg:general} effectively draws on everything developed earlier in \cite{li22}.

These algorithms and their convergence proofs will be given in Section~\ref{sec:iterative}. We will also present a range of numerical experiments arising from real-world applications in Section~\ref{sec:experiments} to illustrate Algorithm~\ref{alg:Schatten}.

\section{Conventions and notations}

We write all vectors in $\R^n$ as column vectors, i.e., $\R^n \equiv \R^{n \times 1}$. When  $x_1, \dots, x_n \in \R$ are enclosed by parentheses, it denotes a column vector, i.e.,
\[
(x_1,\dots,x_n) \coloneqq \begin{bmatrix}
    x_1 \\
    \vdots \\
    x_n
\end{bmatrix}.
\]
When enclosed by square brackets, $[x_1, \dots, x_n] \in \R^{1 \times n}$ is a row vector. 

Throughout this article, when $ A \in \R^{m \times n}$, $B \in \R^{m \times p}$, and $C \in \R^{q \times n}$, we assume without loss of generality that $m\ge p$ and $n\ge q$, since otherwise we may simply add rows of zeros to $A$ and $B$ or columns of zeros to $A$ and $C$ to meet these requirements. We write  $E_{ij} \in \R^{m \times n}$ for the standard basis matrix with $1$ in its $(i,j)$th entry and $0$ elsewhere; when $n = 1$, we write $e_i \coloneqq E_{i1}$ for the standard basis vector. We use upper case letters for matrices and lower case letters for vectors and scalars.

Adhering to convention in numerical linear algebra, the vectorization operator  $\vect : \R^{m \times n} \to \R^{mn}$ is defined by
\[
\vect(X) \coloneqq (x_{11}, \dots, x_{m1}, x_{12}, \dots, x_{m2}, \dots, x_{1n}, \dots, x_{mn}), 
\]
i.e., with respect to the reverse lexicographic ordering on indices. As a result of this we have the somewhat awkward interplay with the Kronecker product:
\begin{equation}\label{eq:vec}
    \vect(BXC) = (C^\tp \otimes B)\vect(X)
\end{equation}
for any $B \in \R^{m \times p}$, $C \in \R^{q \times n}$, and $X \in \R^{p \times q}$. Note that if we had defined $\vect$ with respect to the lexicographic ordering, it would simply have been ``$\vect(BXC) = (B \otimes C)\vect(X)$.''

Let $X \in \R^{m \times n}$. For $i = 1,\dots,\min(m,n)$, we write $\sigma_i(X)$ for the $i$th largest singular value of $X$. We view singular values as a \emph{function} $\sigma: \R^{m \times n} \to \R^{\min(m,n)}$,
\begin{equation}\label{eq:sing}
\sigma(X) \coloneqq \bigl(\sigma_1(X),\dots , \sigma_{\min(m,n)}(X)\bigr).
\end{equation}
Let $X \in \R^{n \times n}$. We view eigenvalues as a \emph{set} $\lambda(X) \subseteq \C^n$. We write $\lambda_{\max}(X)$ and $\lambda_{\min}(X)$ for the largest and smallest eigenvalues by magnitude.
 
For $p \in [1,\infty]$, the vector $p$-norm is denoted by $\lVert\,\cdot\,\rVert_{p}$ and the Schatten $p$-norm is defined by
\begin{equation}\label{eq:sch}
    \lVert X \rVert_{\st,p} \coloneqq \Bigl(\sum\nolimits_{i=1}^{\min(m,n)}\sigma_i(X)^p \Bigr)^{\frac{1}{p}} = \lVert \sigma(X) \rVert_p.
\end{equation}
In particular, $\lVert\,\cdot\,\rVert_{\st,1}$ is the nuclear norm, $\lVert\,\cdot\,\rVert_{\st,\infty}$ is the spectral norm, and $\lVert\,\cdot\,\rVert_{\st,2} = \lVert\,\cdot\,\rVert_\F$ is the Frobenius norm. The letter $p$ is used both as matrix dimensions and as norm powers, but never in close proximity.

Given a finite-dimensional vector space $\mathbb{V}$ with norm $\lVert\,\cdot\,\rVert$, we denote its dual space and dual norm by
\[
\mathbb{V}^* \coloneqq \{\varphi : \mathbb{V} \to \R : \varphi  \text{ is linear}\}, \qquad \lVert \varphi  \rVert^* \coloneqq \sup_{\lVert x \rVert = 1} \varphi (x).
\]

A function $f: \R^{m \times n} \to \R$ is called orthogonally invariant if it satisfies $f(UXV) = f(X)$ for all $X \in \R^{m \times n}$, and orthogonal matrices $U \in \R^{m \times m}$ and $V \in \R^{n \times n}$. For example, rank, singular values, Schatten norms are all orthogonally invariant functions. 

The subdifferential \cite{Clarke} of a function $f:\R^{m \times n} \to \R \cup \{\infty\}$ at $X \in \R^{m \times n}$ is denoted by
\[
\partial f(X) \coloneqq \bigl\{G\in \R^{m \times n} : f(Y)-f(X) \ge \tr\bigl(G^\tp (Y-X) \bigr) \text{ for all } Y \in \R^{m \times n}\bigr\}.
\]
For those unfamiliar with this notion, note that subdifferentials are set-valued; an element of $\partial f(X)$ is called a subgradient.

\section{Closed-form solutions}\label{sec:closed-form}

We begin with the four problems with closed form solutions described in Section~\ref{sec:closed}. Our iterative algorithms in Section~\ref{sec:iterative} would, on occasions, call these as subroutines.

\subsection{Affine generalized matrix nearness}\label{sec:affine}

We first address a simple case for later use. The closed form solution to the least squares problem $\min_{x \in \R^n} \lVert A- bx^\tp \rVert_\F$ is standard, commonly required in mean centering \cite{Marden1995}. It is straightforward to extend this to two-sided variants:
\begin{equation}\label{eq:simple}
\min_{x \in \R^p} \; \lVert A- bx^\tp C \rVert_\F \qquad \text{or} \qquad \min_{x \in \R^p} \; \lVert A- Bx c^\tp \rVert_\F,
\end{equation}
where $A \in \R^{m \times n}$, $B \in \R^{m \times p}$, $C \in \R^{p\times n}$, $b \in \R^m$, $c \in \R^n$. By taking transpose, the two problems are easily seen to be identical and we will use the one on the right. Although a special case of \eqref{eq:near} and follows from the closed-form solution in \cite[Equation~2.1]{li22}, this simple case has a normal equation type closed form solution that avoids singular value decomposition:
\begin{lemma}\label{lem:rank-1}
Let $A \in \R^{m \times n}$, $B \in \R^{m \times p}$, and $0 \ne c \in \R^n$. Then
\[
x_* = \frac{1}{ c^\tp c } B^\dagger A c
\]
gives a solution to \eqref{eq:simple}.
\end{lemma}
\begin{proof}
Since  $\lVert A- Bxc^\tp \rVert_\F^2 = \lVert A \rVert_\F^2 - 2\tr(A^\tp Bxc^\tp) + \lVert Bxc^\tp \rVert_\F^2$, it is equivalent to minimizing 
\[
    f(x) = - 2\tr(A^\tp Bxc^\tp) + \lVert Bxc^\tp \rVert_\F^2 = \lVert c\rVert_{2}^{2}x^\tp B^\tp Bx -2x^\tp B^\tp Ac
\]
Taking gradient  $\nabla f(x) = 2\lVert c \rVert_2^2 B^\tp B x - 2 B^\tp Ac= 0$ gives the required solution.
\end{proof}

Let $A \in \R^{m \times n}$, $B \in \R^{m \times p}$, and $c \in \R^m$. The affine orthogonal Procrustes problem
\[
\min_{X^\tp X = I,\, x\in \R^n} \lVert A- B X + cx^\tp \rVert_\F
\]
is well-known and closely related to the notion of centering matrices in statistics \cite{Marden1995}. The affine generalized matrix nearness problem of this section may be regarded as a generalization to arbitrary constraints beyond orthogonality and allowing for extra matrix multipliers:
\begin{equation}\label{eq:affine}
\min_{X \in \mathcal{S},\, x\in \R^n} \lVert A- B X C + dx^\tp E \rVert_\F \qquad\text{or}\qquad
\min_{X \in \mathcal{S},\, x\in \R^m} \lVert A- B X C + D x e^\tp \rVert_\F.
\end{equation}
Here  $C \in \R^{q \times n}$, $e \in \R^n$, and we assume that $E \in \R^{n \times n}$ and $D \in \R^{m \times m}$ are nonsingular.
Since the Frobenius norm is invariant under taking transpose, the two problems in \eqref{eq:affine} are really one and the same. We will work with the version on the right. The following shows how such an affine generalized matrix nearness problem can be transformed into a (standard) generalized matrix nearness problem.
\begin{proposition}\label{prop:affine}
Let $A \in \R^{m \times n}$, $B \in \R^{m \times p}$, $C \in \R^{q \times n}$, nonsingular $D \in \R^{m \times m}$,  $0 \ne e \in \R^{n}$, and $\mathcal{S} \subseteq \R^{p \times q}$ be a constraint set. If
\begin{align}
X_* &\in \argmin_{X \in \mathcal{S}} \; \biggl\lVert A \biggl(I - \frac{ee^\tp}{\lVert e \rVert_2^2} \biggr) - BXC\biggl(I - \frac{ee^\tp}{\lVert e \rVert_2^2} \biggr) \biggr\rVert_\F, \label{eq:Xstar} \\
x_* &= \frac{1}{e^\tp e} D^\dagger (BX_*C - A) e, \notag
\end{align}
then $(X_*, x_*)$ is a solution to \eqref{eq:affine}.
\end{proposition}
\begin{proof}
By Lemma~\ref{lem:rank-1}, for each $X \in \mathcal{S}$, 
\[
    \frac{1}{\lVert e\rVert_2^2} D^\dagger (BX_*C - A) e \in \argmin_{x\in \R^m} \; \lVert A- B X C + Dx e^\tp \rVert_\F.\]
Therefore, 
\begin{align*}
    \min_{X \in \mathcal{S},\, x\in \R^m} \lVert A- B X C + Dxe^\tp \rVert_\F &= \min_{X \in \mathcal{S}} \; \biggl\lVert A - BXC -  DD^\dagger(A- B X C)\frac{ee^\tp}{\lVert e \rVert_2^2} \biggr\rVert_\F\\
    &= \min_{X \in \mathcal{S}} \; \biggl\lVert A \biggl(I - \frac{ee^\tp}{\lVert e \rVert_2^2} \biggr) - BXC\biggl(I - \frac{ee^\tp}{\lVert e \rVert_2^2} \biggr) \biggr\rVert_\F. \qedhere
\end{align*}
\end{proof}
Note that \eqref{eq:Xstar} takes the form of a generalized matrix nearness problem in \eqref{eq:near}. For the constraints~\ref{it:rank}--\ref{it:pos}, its solution $X_*$ have been found in \cite{li22}.

\subsection{Nearest separable Kronecker product}\label{sec:Kronecker2}

One might wonder what happens when we replace the matrix product in \eqref{eq:near} with Kronecker product, i.e.,
\begin{equation}\label{eq:Kronecker}
    \min_{X \in \mathcal{S}} \; \lVert A- B \otimes X \otimes C\rVert_\F
\end{equation}
for given $A \in \R^{m_1pm_2 \times n_1qn_2}$, $B \in \R^{m_1 \times n_1}$, and $C \in \R^{m_2 \times n_2}$. We show that it may be transformed into a matrix nearness problem.

Let $\mathcal{R}:\R^{m_1m_2p \times n_1n_2q} \to \R^{m_1n_1m_2n_2 \times pq}$ be the unique linear operator satisfying 
\[
\mathcal{R}(Y \otimes X \otimes Z) = \vect(Y \otimes Z) \vect(X)^\tp
\]
for all $Y \in \R^{m_1 \times n_1}$, $X \in \R^{p \times q}$, and $Z \in \R^{m_2 \times n_2}$.  Linear operators of this form are called rearrangement operators (see Section~\ref{sec:Kronecker}), convenient in problems involving Kronecker products. They entail nothing more than a rearrangement of the indices of matrix entries and are therefore invertible and isometries in the Frobenius norm.
\begin{proposition}
Let $A \in \R^{m_1pm_2 \times n_1qn_2}$, $B \in \R^{m_1 \times n_1}$, $C \in \R^{m_2 \times n_2}$, and $\mathcal{S} \subseteq \R^{p \times q}$. If $H \in \R^{p \times q}$ is the unique matrix with
\[
\vect(H) = \frac{1}{\lVert B \rVert_\F^2 \lVert C \rVert_\F^2} \mathcal{R}(A)^\tp \vect(B \otimes C),
\]
then any
\[
X_* \in \argmin_{X \in \mathcal{S}} \; \lVert X - H \rVert_\F
\]
is a solution to \eqref{eq:Kronecker}.
\end{proposition}
\begin{proof}
Since $\mathcal{R}$ preserves Frobenius norms,
\begin{align*}
\lVert A - B &\otimes X \otimes C \rVert^2_\F = \lVert \mathcal{R}(A) - \vect(B \otimes C) \vect(X)^\tp \rVert^2_\F\\
&= \lVert \mathcal{R}(A) \rVert_\F^2 - 2 \vect(B \otimes C)^\tp \mathcal{R}(A) \vect(X) + \lVert B \rVert_\F^2 \lVert C \rVert_\F^2 \lVert X \rVert_\F^2\\
&= \begin{multlined}[t]
    \lVert \mathcal{R}(A) \rVert^2_\F - \frac{\lVert \mathcal{R}(A)^\tp \vect(B \otimes C) \rVert_2^2}{\lVert B \rVert_\F^2 \lVert C \rVert_\F^2}\\
    + \lVert B \rVert_\F^2 \lVert C \rVert_\F^2 \biggl(\lVert X \rVert^2 - \frac{2\mathcal{R}(A)^\tp \vect(B \otimes C)\vect(X)}{\lVert B \rVert_\F^2 \lVert C \rVert_\F^2} + \frac{\lVert \mathcal{R}(A)^\tp \vect(B \otimes C) \rVert_2^2}{\lVert B \rVert_\F^4 \lVert C \rVert_\F^4}\biggr)
    \end{multlined}\\
&=  \lVert \mathcal{R}(A) \rVert^2_\F - \frac{\lVert \mathcal{R}(A)^\tp \vect(B \otimes C) \rVert_2^2}{\lVert B \rVert_\F^2 \lVert C \rVert_\F^2}+ \lVert B \rVert_\F^2 \lVert C \rVert_\F^2  \lVert \vect(X) - \vect(F) \rVert_2^2\\
&= \lVert \mathcal{R}(A) \rVert^2_\F - \frac{\lVert \mathcal{R}(A)^\tp \vect(B \otimes C) \rVert_2^2}{\lVert B \rVert_\F^2 \lVert C \rVert_\F^2}+ \lVert B \rVert_\F^2 \lVert C \rVert_\F^2  \lVert X - H \rVert_2^2. \qedhere
\end{align*}
\end{proof}
Strictly speaking, the problem \eqref{eq:Kronecker} is of a different nature from other problems considered in this article and in \cite{li22}. For instance, the results of Section~\ref{sec:iterative} will not apply to \eqref{eq:Kronecker}. 

\subsection{Nearest Kronecker rank-$r$}\label{sec:Kronecker}

In this section as well as Section~\ref{sec:partial-trace}, any matrix $A \in \R^{m_1m_2 \times n_1n_2}$ will be assumed to be partitioned into a block matrix
\begin{equation}\label{eq:part}
A = (A_{ij}), \quad A_{ij} \in \R^{m_2 \times n_2}, \quad i = 1,\dots, m_1, \; j = 1,\dots, n_1,
\end{equation}
and similarly for matrices in $\R^{m_1 m_2 \times p_1 p_2}$, $\R^{p_1 p_2 \times q_1 q_2}$, and $\R^{q_1 q_2 \times n_1 n_2}$.

Let $A \in \R^{m_1m_2 \times n_1n_2}$, $B_1 \in \R^{m_1 \times p_1}$, $B_2 \in \R^{m_2 \times p_2}$, $C_1 \in \R^{q_1 \times n_1}$, and $C_2 \in \R^{q_2 \times n_2}$. We repeat the definition of Kronecker rank for easy reference: For $X \in \R^{p_1p_2 \times q_1q_2}$,
\[
\rank_\otimes (X) \coloneqq  \min \{ r : X = Y_1 \otimes Z_1 + \dots + Y_r \otimes Z_r\},
\]
with minimum taken over $Y_1,\dots,Y_r\in \R^{p_1 \times q_1}$, $Z_1,\dots,Z_r \in \R^{p_2 \times q_2}$. We seek the solution to
\begin{equation}\label{eq:Kronecker-problem} 
\min_{\rank_\otimes(X) \le r} \lVert  A- (B_1 \otimes B_2)X(C_1 \otimes C_2) \rVert_\F.
\end{equation}
As mentioned in the introduction, without the $B_1 \otimes B_2$ and $C_1 \otimes C_2$ multipliers, the solution may be found in \cite{allen,kon, vanLoan}.

We introduce another rearrangement operator. This terminology is due to \cite{vanLoan}, and is called tilde transform in \cite{kon}. Let $\mathcal{R} : \R^{m_1 m_2 \times n_1 n_2} \to \R^{m_2 n_2 \times m_1 n_1}$ be the unique linear operator satisfying
\begin{equation}\label{eq:tilde-vec}
    \mathcal{R}(Y \otimes Z) = \vect(Z)\vect(Y)^\tp
\end{equation}
for all $Y \in \R^{m_1 \times n_1}$, $Z \in \R^{m_2 \times n_2}$. As we mentioned in the last section, rearrangement operators always preserve Frobenius norms and are invertible. If desired, one may define $\mathcal{R}$ via an explicit expression
\[
\mathcal{R}(X) = \sum_{j=1}^{n_1} \sum_{i=1}^{m_1} \vect(X_{ij})\vect(E_{ij})^\tp,
\]
where $E_{ij} \in \R^{m_1 \times n_1}$, $i = 1,\dots, m_1$, $j = 1,\dots, n_1$, are the standard basis matrices. But this is evidently more cumbersome than defining it via the property we seek, i.e., \eqref{eq:tilde-vec}.

The block vectorization operator is defined in \cite{kon} as
\[\vecb(A) \coloneqq \vect(\mathcal{R}(A))\]
for $A \in \R^{m_1 m_2 \times n_1 n_2}$ with the assumed block partition \eqref{eq:part}. For example, if $m_1 = n_1 = 2$, then
\[
    \vecb\biggl(\begin{bmatrix}
        A_{11} &A_{12}\\
        A_{21} &A_{22}
    \end{bmatrix}\biggr) = \bigl(
        \vect(A_{11}),
        \vect(A_{21}),
        \vect(A_{12}),
        \vect(A_{22}) \bigr).
\]
The operator preserves inner products: For $Y_1, Y_2 \in \R^{m_1m_2 \times n_1n_2}$, 
\[
\vecb(Y_1)^\tp \vecb(Y_2) = \tr(Y_1^\tp Y_2).
\]
It also behaves much like $\vect$ in \eqref{eq:vec}  \cite[Equation~11]{kon}: For any $B \in \R^{m_1m_2 \times p_1p_2}$, $C \in \R^{q_1q_2 \times n_1n_2}$, and $X \in \R^{p_1p_2 \times q_1q_2}$,
\begin{equation}\label{eq:product-rule}
    \vecb(BXC) = \Biggl(\sum_{l=1}^{p_1}\sum_{k=1}^{m_1}\sum_{j=1}^{n_1}\sum_{i=1}^{q_1} E_{ij} \otimes F_{kl} \otimes C_{ji}^\tp \otimes B_{kl} \Biggr) \vecb(X),
\end{equation}
where $\{ E_{ij}  \in \R^{n_1 \times q_1}: i = 1,\dots,q_1, \, j = 1, \dots, n_1\}$ and $\{ F_{kl} \in \R^{m_1 \times p_1}: k=1, \dots, m_1, \, l = 1,\dots,p_1\}$ are the standard bases of the respective spaces.

We now show that the generalized Kronecker rank-$r$ nearness problem reduces to the generalized rank-$r$ nearness problem, whose closed form solution may be found in \cite{Sondermann, Fri, li22}.
\begin{proposition}\label{prop:Kronecker-vecb}
Let $A \in \R^{m_1m_2 \times n_1n_2}$, $B_1 \in \R^{m_1 \times p_1}$, $B_2 \in \R^{m_2 \times p_2}$, $C_1 \in \R^{q_1 \times n_1}$, and $C_2 \in \R^{q_2 \times n_2}$. Then $X_* \in \R^{p_1p_2 \times q_1q_2}$ is a solution to \eqref{eq:Kronecker-problem} if and only if $\mathcal{R}(X_*)$ is a solution to 
\begin{equation} \label{eq:Kronecker-vecb}
    \min_{\rank(\widehat{X})\le r} \lVert \mathcal{R}(A)- (C_2^\tp \otimes B_2)\widehat{X}(C_1 \otimes B_1^\tp) \rVert_\F.
\end{equation}
\end{proposition}
\begin{proof}
Let $C_1 = (c_{ji})$ and $B_1 = (b_{kl})$. Then 
\begin{align*}
    \sum_{l=1}^{p_1}\sum_{k=1}^{m_1}\sum_{j=1}^{n_1}\sum_{i=1}^{q_1}(E_{ij} \otimes E_{kl}) \otimes (c_{ji}C_2^\tp \otimes b_{kl}B_2) &= \sum_{l=1}^{p_1}\sum_{k=1}^{m_1}\sum_{j=1}^{n_1}\sum_{i=1}^{q_1}(c_{ji}E_{ij} \otimes b_{kl}E_{kl}) \otimes (C_2^\tp \otimes B_2)\\
&= C_1^\tp \otimes B_1 \otimes C_2^\tp \otimes B_2.
\end{align*}
Since $\vecb$ preserves Frobenius norms, by \eqref{eq:product-rule} and \eqref{eq:vec},
\begin{align}
\lVert A- (B_1 &\otimes B_2) X(C_1 \otimes C_2)\rVert_\F \notag \\
    &= \bigl\lVert \vecb(A)- \vecb\bigl((B_1 \otimes B_2)X(C_1 \otimes C_2)\bigr)\bigr\rVert_\F \notag \\
    &=  \biggl\lVert \vecb(A)- \biggl(\sum_{l=1}^{p_1}\sum_{k=1}^{m_1}\sum_{j=1}^{n_1}\sum_{i=1}^{q_1}(E_{ij} \otimes E_{kl}) \otimes (c_{ji}C_2^\tp \otimes b_{kl}B_2) \biggr)\vecb(X)\biggr\rVert_\F \notag \\
    &= \lVert \vect(\mathcal{R}(A))- (C_1^\tp \otimes B_1 \otimes C_2^\tp \otimes B_2)\vect(\mathcal{R}(X))\rVert_\F \notag \\
    &= \bigl\lVert \vect(\mathcal{R}(A))- \vect\bigl((C_2^\tp \otimes B_2)\mathcal{R}(X)(C_1 \otimes B_1^\tp)\bigr) \bigr\rVert_\F \notag \\
    &= \lVert \mathcal{R}(A) - (C_2^\tp \otimes B_2)\mathcal{R}(X)(C_1 \otimes B_1^\tp)\rVert_\F. \label{eq:vecb-equiv}
\end{align}
By \eqref{eq:tilde-vec}, $\rank_\otimes(X) \le r $ if and only if $\rank(\mathcal{R}(X)) \le r$. Writing $\widehat{X} = \mathcal{R}(X)$ yields \eqref{eq:Kronecker-vecb}.
\end{proof}

We will use the solution in \cite{li22}, simplest and most explicit among \cite{Sondermann, Fri, li22}, to describe an algorithm. We emphasize that one does not need to compute the singular value decomposition of $C_1^\tp \otimes B_1$ and $C_2^\tp \otimes B_2$ in Algorithm~\ref{alg:Kronecker}, i.e., of $m_1n_1 \times p_1q_1$ and $m_2n_2 \times p_2q_2$ matrices. These follow from the singular value decompositions of $B_1$, $C_1$, $B_2$, $C_2$ since Kronecker product preserves singular value decompositions.

\begin{algorithm}[htb!]
    \caption{Generalized Kronecker rank-$r$ nearness} \label{alg:Kronecker}
    \begin{algorithmic}[1]
    \Require $A \in \R^{m_1m_2 \times n_1n_2}$, $B_1 \in \R^{m_1 \times p_1}$, $B_2 \in \R^{m_2 \times p_2}$, $C_1 \in \R^{q_1 \times n_1}$, $C_2 \in \R^{q_2 \times n_2}$;
    \State compute $\mathcal{R}(A)$;
    \State compute singular value decompositions  of $B_1$, $C_1$, $B_2$, $C_2$ to get
    \[(C_1^\tp \otimes B_1) = U_1 \begin{bmatrix}
        \Sigma_1 &0\\ 0 &0
    \end{bmatrix} V_1^\tp,\quad (C_2^\tp \otimes B_2) = U_2 \begin{bmatrix}
        \Sigma_2 &0\\ 0 &0
    \end{bmatrix} V_2^\tp;\]
    \State compute $A_{11}$ from 
    \[U_2^\tp \mathcal{R}(A)V_1 = \begin{bmatrix}
        A_{11} & A_{12} \\
        A_{21} & A_{22}
    \end{bmatrix};\]
    \State compute singular value decomposition $A_{11} = U\Sigma V^\tp$;
    \State compute 
    \[\widehat{X} = V_2 \begin{bmatrix}
        \Sigma_2^{-1} U_r\Sigma_r V_r^\tp \Sigma_1^{-1} &0\\ 0 &0
    \end{bmatrix} U_1^\tp\]
    where $U_r$ is the first $r$ columns of $U$, $V_r$ those of $V$, and $\Sigma_r$ the top left $r\times r$ block of $\Sigma$;
    \State compute $X_* =  \mathcal{R}^{-1}(\widehat{X}) \in \R^{p_1p_2 \times q_1q_2}$;
    \Ensure $X_*$.
    \end{algorithmic}
\end{algorithm}

\subsection{Nearest prescribed partial trace}\label{sec:partial-trace}

The use of Kronecker product also allows us to discuss notions with no nontrivial counterpart for matrix product. One example is partial trace, simple but ubiquitous in quantum computing \cite{NielsenChuang}, indispensable for describing observable quantities of subsystems in a composite system.

The partial trace of a matrix $X \in \R^{p^2 \times p^2}$, partitioned as $X = (X_{ij})$ with blocks $X_{ij} \in \R^{p \times p}$, $i,j = 1,2,\dots,p$, is
\[
\tr_\Par(X) = \sum_{i=1}^{p} \tr(X_{ii}),
\]
i.e., it is a matrix-valued trace formed by taking sum of diagonal blocks and is compatible with Kronecker product:
\[
\tr_\Par(Y \otimes Z) = \tr(Y)Z
\]
for any $Y$ and $Z \in \R^{p \times p}$.

We will consider the generalized  nearness problem for prescribed partial trace, i.e., its value $\lambda$ is prescribed a priori:
\begin{equation}\label{eq:partial-trace}
    \min \; \{\lVert  A- (B_1 \otimes B_2)X(C_1 \otimes C_2)  \rVert_\F: \tr_\Par(X(W \otimes I)) = \lambda W \text{ for some } W \in \R^{p \times p}\},
\end{equation}
for fixed $A \in \R^{m_1m_2\times n_1n_2}$, $B_1 \in \R^{m_1 \times p}$, $B_2 \in \R^{m_2 \times p}$, $C_1 \in \R^{p \times n_1}$, $C_2 \in \R^{p \times n_2}$, and $\lambda \in \R$. This reminds us of the prescribed eigenvalue problem in \ref{it:eig} on p.~\pageref{it:eig}, although the latter is not a special case of \eqref{eq:partial-trace}. Nevertheless, our first method of solving \eqref{eq:partial-trace} is by a reduction to the prescribed eigenvalue problem.

\begin{proposition}\label{prop:partial-trace}
Let $A \in \R^{m_1m_2\times m_1m_2}$, $B_1 \in \R^{m_1 \times p}$, $B_2 \in \R^{m_2 \times p}$, $C_1 \in \R^{p \times n_1}$, $C_2 \in \R^{p \times n_2}$, and $\lambda \in \R$. Then $X_* \in \R^{p^2 \times p^2}$ is a solution to \eqref{eq:partial-trace} if and only if $\mathcal{R}(X_*)$ is a solution to 
\begin{equation} \label{eq:partial-trace-eig}
    \min_{\lambda \in \lambda(\widehat{X})} \lVert \mathcal{R}(A)- (C_2^\tp \otimes B_2)\widehat{X}(C_1 \otimes B_1^\tp) \rVert_\F.
\end{equation}
\end{proposition}

\begin{proof}
By the singular value decomposition of $\mathcal{R}(X)$, we have $\mathcal{R}(X) = \sum_{i=1}^{r} \vect(Z_i) \vect(Y_i)^\tp$ for some $Y_i \in \R^{p_1 \times p_1}$ and $Z_i \in \R^{p_2 \times p_2}$, $i=1,\dots,r \le \min\{p_1^2,p_2^2\}$. For any $W \in \R^{p_1 \times p_1}$, 
\[\mathcal{R}(X)\vect(W) = \sum_{i=1}^{r} \vect(Z_i) \vect(Y_i)^\tp \vect(W) =  \sum_{i=1}^{r} \tr(Y_i W) \vect(Z_i).\]
On the other hand, by \eqref{eq:tilde-vec}, $X = \sum_{i=1}^{r} Y_i \otimes Z_i$, so
\[
\tr_\Par\bigl(X(W \otimes I)\bigr) =\tr_\Par \biggl(\sum_{i=1}^{r} (Y_i \otimes Z_i) (W \otimes I) \biggr) = \sum_{i=1}^{r}  \tr_\Par \bigl((Y_i W \otimes Z_i)\bigr)= \sum_{i=1}^{r} \tr(Y_i W) Z_i.
\]
Therefore, 
\begin{equation}\label{eq:partial-eig}
    \mathcal{R}(X)\vect(W) = \lambda \vect(W) \quad \text{if and only if} \quad \tr_\Par(X(W \otimes I)) = \lambda W.
\end{equation}
The rest of the proof follows from the same calculation in \eqref{eq:vecb-equiv}.
\end{proof}

Our second method of solving \eqref{eq:partial-trace} uses a further reduction to the generalized rank-constrained problem \eqref{eq:Kronecker-problem}, which could in turn be solved by Algorithm~\ref{alg:Kronecker}.
\begin{corollary}\label{cor:partial-trace}
Let $A \in \R^{m_1m_2\times m_1m_2}$, $B_1 \in \R^{m_1 \times p}$, $B_2 \in \R^{m_2 \times p}$, $C_1 \in \R^{p \times n_1}$, $C_2 \in \R^{p \times n_2}$, and $\lambda \in \R$. Then $X_* \in \R^{p^2 \times p^2}$ is a solution to \eqref{eq:partial-trace} if and only if $\mathcal{R}(X_*)$ is a solution to 
\begin{equation} \label{eq:partial-trace-rank}
\min_{\rank(\widehat{X}-\lambda I) \le p^2-1}
	\lVert \mathcal{R}(A)-\lambda(C_2^\tp C_1 \otimes B_2 B_1^\tp) - (C_2^\tp \otimes B_2)(\widehat{X} - \lambda I)(C_1 \otimes B_1^\tp)\rVert_\F.
\end{equation}
\end{corollary}
\begin{proof}
Observe that $\lambda \in \lambda(X)$ if and only if $\rank(X-\lambda I) \le p^2-1$. Therefore, the minimum in  \eqref{eq:partial-trace-eig} equals that in \eqref{eq:partial-trace-rank}.
\end{proof}

One might wonder if the more general variant of \eqref{eq:Kronecker-problem},
\begin{gather*}
\min \; \{\lVert  A-  BXC  \rVert_\F: \tr_\Par(X(W \otimes I)) = \lambda W \text{ for some } W \in \R^{p \times p}\},
\intertext{as well as that of \eqref{eq:partial-trace} from the last section,}
\min_{\rank_\otimes(X) \le r} \lVert  A- BXC \rVert_\F,
\end{gather*}
would similarly have closed form solutions. Note that here we no longer require that $B = B_1 \otimes B_2$ and $C = C_1 \otimes C_2$. We do not have an answer and leave these open questions for interested readers. Nevertheless, the iterative Algorithm~\ref{alg:Schatten} in Section~\ref{sec:iterative} will solve not only these problems, but with any Schatten norm in place of the Frobenius norm, to any desired degree of accuracy.

\section{Norm dependence of generalized nearness problems}\label{sec:norm}

The closed form solutions in Section~\ref{sec:closed-form}, as well as those in \cite{li22}, are exclusively for generalized nearness problems in the Frobenius norm $\lVert \, \cdot \, \rVert_\F$. From this section onwards, our focus would be on arbitrary orthogonally invariant norms $\lVert\,\cdot\,\rVert$, occasionally restricted to Schatten $p$-norms $\lVert \, \cdot \, \rVert_{\st, p}$, $p \in [1,\infty]$.

As we mentioned in Section~\ref{sec:imposs}, we will show that one cannot expect a Mirsky-type result \cite[Theorem~2]{mir} for rank constrained generalized nearness problems. Let $ A\in \R^{m \times n}$, $B \in \R^{m \times p}$, and $C \in \R^{q \times n}$. Let
\[
B=U_B\begin{bmatrix}
    \Sigma_B & 0 \\
0 & 0
\end{bmatrix} V_B^\tp, \quad C=U_C\begin{bmatrix}
    \Sigma_C & 0 \\
0 & 0
\end{bmatrix} V_C^\tp,
\]
be the singular value decompositions of $B$ and $C$ with $\Sigma_B = \diag(\sigma_1(B),\dots,\sigma_{s}(B))$, $\Sigma_C = \diag(\sigma_1(C),\dots,\sigma_{t}(C))$, $\rank(B) = s$, and $\rank(C) = t$. For an orthogonal invariant norm $\lVert\,\cdot\,\rVert$,
\[
    \lVert A-BXC \rVert = \biggl \lVert U^\tp_B A V_C - \begin{bmatrix}
    \Sigma_B & 0 \\
0 & 0
\end{bmatrix} V_B^\tp X  U_C\begin{bmatrix}
   \Sigma_C & 0 \\
0 & 0
\end{bmatrix}\biggr\rVert.
\]
Let
\begin{equation}\label{eq:equiv1}
    U^\tp_B A V_C \eqqcolon \begin{bmatrix}
        A_{11} & A_{12} \\
    A_{21} & A_{22}
    \end{bmatrix}, \quad 
    V^\tp_B X U_C \eqqcolon \begin{bmatrix}
        X_{11} & X_{12} \\
    X_{21} & X_{22}
    \end{bmatrix}
\end{equation}
with $A_{11},X_{11}\in \R^{s \times t}$. Then
\begin{equation}\label{eq:equiv2}
    \lVert A-BXC \rVert = \biggl\lVert \begin{bmatrix}
        A_{11}-\Sigma_B X_{11}\Sigma_C & A_{12} \\
    A_{21} & A_{22}
    \end{bmatrix} \biggr\rVert.
\end{equation}
So
\begin{multline}\label{eq:completion}
\begin{bmatrix}
    Y_* & A_{12} \\
A_{21} & A_{22}
\end{bmatrix}  \in
    \argmin_{Y \in \R^{s \times t}} \biggl\lVert \begin{bmatrix}
    Y & A_{12} \\
A_{21} & A_{22}
\end{bmatrix} \biggr\rVert
\\
\Longleftrightarrow
\quad
 V_B
\begin{bmatrix}
    \Sigma_B^{-1}(A_{11}-Y_*)\Sigma_C^{-1} & 0 \\
0 & 0
\end{bmatrix}
U_C^\tp \in
    \argmin_{X \in \R^{p \times q}} \lVert A- BXC\rVert.
\end{multline}
If the left side of \eqref{eq:completion} admits different minimizer under different orthogonally invariant norms, then so does the right side. We now prove an optimality condition for the left side of \eqref{eq:completion}. While the equivalence in \eqref{eq:completion} requires an  orthogonally invariant norm, Proposition~\ref{prop:complete} below holds for any arbitrary norm. 
\begin{proposition}\label{prop:complete}
Let $\lVert\,\cdot\,\rVert$ be a norm on $\R^{m\times n}$ and $\lVert\,\cdot\,\rVert^*$ be the dual norm. For any $S \in \R^{m \times n}$, we write $\varphi_S : \R^{m \times n} \to \R$, $X \mapsto \tr(S^\tp X)$, for the linear functional defined by $S$. Let $A_{12}\in \R^{s \times (n-t)}$, $A_{21}\in \R^{(m-s) \times t}$, $A_{22}\in \R^{(m-s) \times (n-t)}$, and $Y_* \in \R^{s \times t}$. 
Then
\[
\biggl\lVert \begin{bmatrix}
    Y_* & A_{12} \\
A_{21} & A_{22}
\end{bmatrix} \biggr\rVert \le \biggl\lVert \begin{bmatrix}
    Y & A_{12} \\
A_{21} & A_{22}
\end{bmatrix} \biggr\rVert \quad \text{for all } Y \in \R^{s \times t}
\]  if and only if there exists 
\begin{equation}\label{eq:R_form}
    S= \begin{bmatrix}
    0 & S_{12} \\
S_{21} & S_{22}
\end{bmatrix} \in \R^{m \times n}
\end{equation}
with $S_{12}\in \R^{s \times (n-t)}$, $S_{21}\in \R^{(m-s) \times t}$, $S_{22}\in \R^{(m-s) \times (n-t)}$ such that $\lVert \varphi_S \rVert^*=1$ and 
\begin{equation}\label{eq:completion_thm}
    \biggl\lVert \begin{bmatrix}
        Y_* & A_{12} \\
    A_{21} & A_{22}
    \end{bmatrix} \biggr\rVert = \tr\biggl(S^\tp \begin{bmatrix}
        Y_* & A_{12} \\
    A_{21} & A_{22}
    \end{bmatrix} \biggr) = \tr \biggl( S^\tp  \begin{bmatrix}
        0 & A_{12} \\
    A_{21} & A_{22}
    \end{bmatrix}  \biggr).
\end{equation}
\end{proposition}

\begin{proof}
Let $\beta \coloneqq \bigl \lVert \bigl[\begin{smallmatrix}
    Y_* & A_{12} \\
A_{21} & A_{22}
\end{smallmatrix}\bigr]  \bigr\rVert$. By its definition, $\beta \le \bigl \lVert \bigl[\begin{smallmatrix}
    Y & A_{12} \\
A_{21} & A_{22}
\end{smallmatrix}\bigr]  \bigr\rVert$ for all $Y \in \R^{s \times t}$. Consider the subspace
\[
\mathbb{V} = \biggl\{ \begin{bmatrix}
    Y &zA_{12}\\
    zA_{21} &zA_{22}
\end{bmatrix} \in \R^{m \times n} : Y \in \R^{s\times t},\, z \in \R  \biggr\}
\]
and the linear functional $\varphi : \mathbb{V} \to \mathbb{R}$ defined by
\[
\varphi\biggl(\begin{bmatrix} 
    Y &zA_{12}\\
    zA_{21} &zA_{22}
\end{bmatrix}\biggr) = \beta z
\]
for all $Y \in \R^{s\times t}$ and $z \in \R$. 
We must also have $\lVert \varphi\rVert^* \le 1$ as
\[
\biggl\lVert \begin{bmatrix}
    Y &zA_{12}\\
    zA_{21} &zA_{22}
\end{bmatrix} \biggr\rVert = \lvert z \rvert \biggl\lVert \begin{bmatrix}
    Y/z &A_{12}\\
    A_{21} & A_{22}
\end{bmatrix} \biggr\rVert \ge \lvert  \beta z\rvert =  \biggl\lvert \varphi \biggl( \begin{bmatrix} 
    Y &zA_{12}\\
    zA_{21} &zA_{22}
\end{bmatrix} \biggr) \biggr\rvert
\]
for all $Y\in \R^{s \times t}$  and $0 \ne z \in \R$.
In fact, $\lVert  \varphi\rVert^* = 1$ since $\beta =  \varphi\bigl(  \bigl[\begin{smallmatrix}
    Y_* & A_{12} \\
A_{21} & A_{22}
\end{smallmatrix}\bigr] \bigr) \le \lVert \varphi \rVert^* \beta$. By Hahn--Banach Theorem, $\varphi$ can be extended to a linear functional on $\R^{m\times n}$ with the same norm. By Riesz representation, any such linear functional must take the form $\varphi_S$ for some $S  \in \R^{m\times n}$.  So  $\lVert \varphi_S\rVert^* = 1$ and $\varphi_S \bigr\rvert_{\mathbb{V}} =\varphi$, i.e.,
\[
\varphi_S\biggl( \begin{bmatrix}
    Y &zA_{12}\\
    zA_{21} &zA_{22}
\end{bmatrix} \biggr) =
\tr\biggl( S^\tp \begin{bmatrix}
    Y &zA_{12}\\
    zA_{21} &zA_{22}
\end{bmatrix} \biggr)=\beta z
\]
for all $Y\in \R^{s \times t}$ and $z \in \R$. Since 
$
\tr \bigl( S^\tp 
\bigl[\begin{smallmatrix}
    Y &0\\
    0 &0
\end{smallmatrix}\bigr] \bigr) = 0
$
for all $Y \in \R^{s \times t}$, the matrix $S$ must be of the form \eqref{eq:R_form} and satisfies \eqref{eq:completion_thm}.

Now, suppose that there exists $S\in \R^{m \times n}$ of the form \eqref{eq:R_form} such that $\lVert \varphi_S \rVert^*=1$ and satisfies \eqref{eq:completion_thm}. Then for any $Y\in \R^{s \times t}$,
\begin{align*}
    \biggl\lVert  \begin{bmatrix}
        Y_* & A_{12} \\
    A_{21} & A_{22}
    \end{bmatrix}  \biggr\rVert &= \biggl\lvert \tr\biggl(S^\tp  \begin{bmatrix}
        Y_* & A_{12} \\
    A_{21} & A_{22}
    \end{bmatrix} \biggr) \biggr\rvert \\
    &= \biggl\lvert \tr\biggl(S^\tp  \begin{bmatrix}
        Y & A_{12} \\
    A_{21} & A_{22}
    \end{bmatrix} \biggr) \biggr\rvert\le \lVert S\rVert^* \biggl\lVert  \begin{bmatrix}
        Y & A_{12} \\
    A_{21} & A_{22}
    \end{bmatrix}  \biggr\rVert = \biggl\lVert  \begin{bmatrix}
        Y & A_{12} \\
    A_{21} & A_{22}
    \end{bmatrix}  \biggr\rVert. \qedhere
\end{align*}
\end{proof}
An alternative way, perhaps slightly shorter, to prove Proposition~\ref{prop:complete} would be to deduce it from \cite[Theorem~1.1]{Singer}; but we prefer a direct proof via the much better known Hahn--Banach Theorem.

With Proposition~\ref{prop:complete} in place, we may now construct an example to show that Mirsky Theorem \cite[Theorem~2]{mir} does not extend to the generalized nearness problem.
\begin{example}\label{example:completion}
Consider the best rank-one generalized matrix nearness problem with
\[
A = \begin{bmatrix}
    0 &1 \\
    1 &1
\end{bmatrix}, \qquad B = \begin{bmatrix}
    -1 &0 \\0 &0
\end{bmatrix}, \qquad C = \begin{bmatrix}
    1 &0 \\0 &0
\end{bmatrix},
\]
which gives\
\begin{equation}\label{eq:con}
\min_{\rank(X) \le 1} \; \biggl\lVert \begin{bmatrix}
    0 & 1 \\
    1 & 1
\end{bmatrix} - \begin{bmatrix}
    -1 &0 \\0 &0
\end{bmatrix} X \begin{bmatrix}
    1 & 0 \\0 &0
\end{bmatrix}\biggr\rVert_{\st, p} = \min_{\rank(X) \le 1} \; \biggl\lVert \begin{bmatrix}
    x_{11} & 1 \\
    1 & 1
\end{bmatrix} \biggr\rVert_{\st, p}.
\end{equation}
In case this is not clear, in the minimization problem on the right of \eqref{eq:con}, the matrix $X$ in the rank constraint is \emph{not} the same matrix as the one in the norm objective, likewise in \eqref{eq:uncon} below. 
Recall that the dual norm of the Schatten $p$-norm is the Schatten $q$-norm where $1/p+1/q=1$. By Proposition~\ref{prop:complete}, we know that $X_*  = \bigl[\begin{smallmatrix}
x_* & y_* \\
z_* & w_*
\end{smallmatrix} \bigr]$ is a global minimizer of
\begin{equation}\label{eq:uncon}
\min_{X \in \R^{2 \times 2}} \; \biggl\lVert \begin{bmatrix}
    x_{11} & 1 \\
    1 & 1
\end{bmatrix} \biggr\rVert_{\st, p}
\end{equation}
for $p \in \{1, 2,\infty\}$ if there exists $S_p = \bigl[\begin{smallmatrix}
    0 & b_p \\
c_p & d_p
\end{smallmatrix} \bigr]$ such that $\lVert \varphi_{S_p} \rVert_{\st, q}=1$ 
and 
\[
\tr(S_p^\tp X_*) = 
\tr\biggl(\begin{bmatrix}
    0 & b_p \\
c_p & d_p
\end{bmatrix}^\tp \begin{bmatrix}
    x_* &1\\ 1 &1
\end{bmatrix} \biggr) = \biggl\lVert \begin{bmatrix}
    x_* &1\\ 1 &1
\end{bmatrix} \biggr\rVert_{\st, p}.
\]
We now show that the minima in  \eqref{eq:con} and \eqref{eq:uncon} are the same, i.e., the minimum in \eqref{eq:uncon} can be attained with rank-one matrices. For $p=1$, notice that  
\[
    \biggl\lVert \begin{bmatrix}
        1 &1\\ 1 &1
    \end{bmatrix} \biggr\rVert_{\st,1} = \tr\biggl(\begin{bmatrix}
        0 &1\\
        1 &0
    \end{bmatrix} \begin{bmatrix}
        1 &1\\ 1 &1
    \end{bmatrix} \biggr) =2, \quad \text{where } S_1 = \begin{bmatrix}
        0 &1\\
         1 &0
     \end{bmatrix}
\]
satisfies $\lVert S_1\rVert_{\st,\infty} =1$. So any rank-one matrix of the form $X_*  = \bigl[\begin{smallmatrix}
1 & y_* \\
z_* & w_*
\end{smallmatrix} \bigr]$ is a solution to \eqref{eq:uncon}.

For $p=2$, notice that  
\[
    \biggl\lVert \begin{bmatrix}
        0 &1\\ 1 &1
    \end{bmatrix} \biggr\rVert_{\st,2} = \tr\Biggl(\begin{bmatrix}
    0 &\frac{1}{\sqrt{3}}\\
    \frac{1}{\sqrt{3}} &\frac{1}{\sqrt{3}}
\end{bmatrix} \begin{bmatrix}
    0 &1\\ 1 &1
\end{bmatrix} \Biggr) = \sqrt{3}, \quad \text{where }  S_2 = \begin{bmatrix}
    0 &\frac{1}{\sqrt{3}}\\
    \frac{1}{\sqrt{3}} &\frac{1}{\sqrt{3}}
\end{bmatrix}
\]
satisfies $\lVert S_2\rVert_{\st,2} =1$. So any rank-one matrix of the form $X_*  = \bigl[\begin{smallmatrix}
0 & y_* \\
z_* & w_*
\end{smallmatrix} \bigr]$ is a solution to \eqref{eq:uncon}.

Similarly, for $p=\infty$, 
\[
    \biggl\lVert \begin{bmatrix}
        -1 &1\\ 1 &1
    \end{bmatrix} \biggr\rVert_{\st,\infty} = \tr\Biggl(\begin{bmatrix}
        0 &\frac{1}{2\sqrt{2}}\\
        \frac{1}{2\sqrt{2}} &\frac{1}{\sqrt{2}}
    \end{bmatrix} \begin{bmatrix}
        -1 &1\\ 1 &1
    \end{bmatrix} \Biggr) = \sqrt{2} \quad \text{where } S_\infty = \begin{bmatrix}
        0 &\frac{1}{2\sqrt{2}}\\
        \frac{1}{2\sqrt{2}} &\frac{1}{\sqrt{2}}
    \end{bmatrix}
\]
satisfies $\lVert S_\infty\rVert_{\st,1} =1$. So any rank-one matrix of the form $X_*  = \bigl[\begin{smallmatrix}
-1 & y_* \\
z_* & w_*
\end{smallmatrix} \bigr]$ is a solution to \eqref{eq:uncon}.

The minimizers for the $p =1$, $2$, $\infty$ cases are all distinct. We conclude that the solution of
\[
\min_{\rank(X) \le r} \lVert  A- BXC \rVert_{\st,p}
\]
depend on the value of $p$; Mirsky Theorem does not hold here.
\end{example}

\section{Iterative solutions}\label{sec:iterative}

Given that we do not have closed form solutions for norms other than the Frobenius norm, we now turn to iterative solutions. As mentioned in Section~\ref{sec:algo}, our proposed iterative algorithm is very versatile: it works with any of the constraints \ref{it:rank}--\ref{it:pos} listed on p.~\pageref{it:pos} (originally in \cite{li22}); it works with the Kronecker rank constraint in Section~\ref{sec:Kronecker} and partial trace constraint in Section~\ref{sec:partial-trace} --- new constraints not considered in \cite{li22}; it also works with a combination of any of these constraints; and in all cases it works with the more general affine objective function in Section~\ref{sec:affine}.

The iterative algorithm in question is Algorithm~\ref{alg:Schatten}, and it has convergence guarantees: (a) to the global minimizer, (b) from any starting point, (c) at sublinear rate. In fact these hold true more generally for any subdifferentiable $f: \R^{m \times n} \to \R$ that attains the minimum 
\begin{equation}\label{eq:coercive}
    \min_{X \in \mathcal{S}} \; f(A - BXC).
\end{equation}
Since treating this more general case requires no extra effort, we will establish our convergence results in Section~\ref{sec:conv} for any subdifferentiable loss function $f$ in a corresponding Algorithm~\ref{alg:general} that reduces to Algorithm~\ref{alg:Schatten} when $f = \lVert\,\cdot\,\rVert_{\st, p}$.

A word about the sublinear convergence in case it gives the reader pause: Such is the reality of working with norms that are not strongly convex. Note that the Schatten norm $\lVert\,\cdot\,\rVert_{\st, p}$ is convex but not strictly convex for $p =1$ and $\infty$; it is strictly but not strongly convex for all $p \in (1,2)\cup (2,\infty)$; and it is strongly convex for the unique value $p = 2$, the Frobenius norm. Strong convexity is the key ingredient required for linear convergence in \cite[Theorem~4.4]{li22} and we lack that here. And we are not alone: sublinear convergence is typical of works involving $\lVert\,\cdot\,\rVert_{\st, 1}$, the nuclear norm \cite{sublinear2, sublinear1}; without additional assumptions like the Kurdyka--\L ojasiewicz property (which does not hold in our case), a sublinear rate is the best one can do \cite{strongconvex}. Our silver lining is that, in practice, the numerical experiments in Section~\ref{sec:experiments} indicate near linear convergence in Algorithm~\ref{alg:Schatten}.

\subsection{Convergence analysis}\label{sec:conv}

The idea that motivates our algorithm for \eqref{eq:coercive} dates back to \cite{Dyk}. To solve \eqref{eq:coercive} for a general $f$ given that we already know how to solve it for the special $f = \lVert\,\cdot\,\rVert_\F$, a conceivable strategy is to alternate between two subproblems:
\[
\min_{\lVert Y-M \rVert_\F \le \frac{\mu}{2}} f(Y) \qquad \text{and} \qquad \min_{X \in \mathcal{S}} \; \lVert Y - (A - BXC) \rVert_\F
\]
for fixed $\mu > 0 $ and $M \in \R^{m\times n}$. The alternation trades off between minimizing $f$ and staying close to $\mathcal{S}$. However, this simplistic approach does not work: the iterates could end up oscillating between two fixed points. Fortunately, there is also a simple fix: introducing a \emph{Dykstra correction} term $\Delta \in \R^{m \times n}$ \cite{Dyk}.  The result is Algorithm~\ref{alg:general}.

\begin{algorithm}
\caption{Dykstra iteration} \label{alg:general}
\begin{algorithmic}[1]
\Require $f: \R^{m \times n} \to \R$, $ A \in \R^{m \times n}$, $B \in \R^{m \times p}$, $C \in \R^{q \times n} $, $\mu > 0$, $\epsilon > 0$;
\State initialize $X_0, Y_0, \Delta_0, k = 0$;
\State solve for
\begin{equation}\label{eq:X_step}
    X_{k+1} \in \argmin \bigl\{ \lVert Y_{k} - A + BXC - \Delta_k \rVert_\F^2 : X \in \mathcal{S}\bigr\};
\end{equation}
\label{line:X_step}
\State solve for
\begin{equation}\label{eq:Y_step}
    Y_{k+1} \in \argmin  \bigl\{ f(Y) + \tfrac{\mu}{2} \lVert Y - A + BX_{k+1}C - \Delta_k \rVert_\F^2 : Y \in \R^{m \times n}\bigr\};
\end{equation}
\label{line:Y_step}
\State compute 
\begin{equation}\label{eq:D_step}
    \Delta_{k+1} = \Delta_k - Y_{k+1} + A - BX_{k+1}C;
\end{equation}
\State $k = k+1$ and go to line~\ref{line:X_step}.
\end{algorithmic}
\end{algorithm}

The $X_k$-update step \eqref{eq:X_step} is a generalized matrix nearness problem in the Frobenius norm. We may solve it for any $\mathcal{S}$ defined by the constraints~\ref{it:rank}--\ref{it:pos} or a mixture of them via the methods in \cite{li22}, the Kronecker rank constraint by Proposition~\ref{prop:Kronecker-vecb}, and the prescribed partial trace constraints by Proposition~\ref{prop:partial-trace}. Algorithm~\ref{alg:general} may be extended to \eqref{eq:coercive} with an extra affine term:
\[
    \min_{X \in \mathcal{S}} \; f(A - BXC - dx^\tp E) \qquad \text{and} \qquad     \min_{X \in \mathcal{S}} \; f(A - BXC - Dxe^\tp)
\]
by applying Proposition~\ref{prop:affine} to the resulting minimization \eqref{eq:X_step}, which now takes the form in \eqref{eq:affine}.

The $Y_k$-update step \eqref{eq:Y_step} for an arbitrary loss function $f$ is often called a proximal operator in the optimization literature \cite{proximal}. However our goal is to avoid using general purpose optimization methods altogether (otherwise we might as well just solve \eqref{eq:coercive} as a nonlinear optimization problem), ideally without computing a single gradient or subgradient.\footnote{We use them in statements and proofs but never in actual algorithms.} We will show in Section~\ref{sec:Schatten} that for $f = \lVert \, \cdot \, \rVert_{\st,p}$, we have closed-form solutions for $Y_{k+1}$.

We now prove that the iterates generated by Algorithm~\ref{alg:general} always converge at a sublinear rate to the global solution of \eqref{eq:coercive} for any initialization. We require the following two lemmas.
\begin{lemma}[Optimality condition]\label{lem:opt}
Let $f: \R^{m \times n} \to \R$ be subdifferentiable and attain its minimum on $\widehat{\mathcal{S}} \coloneqq \{A - BXC \in \R^{m \times n} : X \in \mathcal{S}\}$. Then $X_* \in \R^{p \times q}$ is a minimizer of \eqref{eq:coercive} if and only if there exists $G_* \in \partial f(A - BX_* C)$ such that
\begin{equation}\label{eq:min}
\tr\bigl(G_*^{\tp}  B(X_* - X)C \bigr) \ge 0 \quad \text{for all } X \in \mathcal{S}.
\end{equation}
\end{lemma}
\begin{proof}
Let $\eta: \R^{m \times n} \to \R \cup \{\infty\}$ be the indicator function given by
\[
\eta(Y) = \begin{dcases}
    0 & Y \in \widehat{\mathcal{S}},\\
    \infty & Y \not\in \widehat{\mathcal{S}}.
\end{dcases}
\]
We rewrite \eqref{eq:coercive} as
\[\min_{X \in \mathcal{S}} \; f(A - BXC) = \min_{Y \in \R^{m \times n}}  f(Y) + \eta(Y).\]
Since $\eta$ is subdifferentiable:
\[
0 \in \partial \eta(Y) = \{G \in \R^{m \times n} : \tr\bigl(G^\tp (Z-Y)\bigr) \le 0 \text{ for all } Z \in \widehat{\mathcal{S}} \}
\]
for all $Y \in \widehat{\mathcal{S}}$, the sum $f + \eta$ is subdifferentiable on $\widehat{\mathcal{S}}$. The function $f + \eta$ attains its minimum at $Y_*$ if and only if $0 \in \partial (f + \eta)(Y_*) = \partial f(Y_*) + \partial \eta(Y_*)$, i.e., there exists $G \in \partial f(Y_*)$ such that $-G \in \partial \eta(Y_*)$. This is equivalent to the existence of $G_* \in \partial f(Y_*)$ with
\begin{equation}\label{eq:minimizer}
    \tr\bigl(G^\tp (Y-Y_*) \bigr)\ge 0 \quad \text{for all } Y \in \widehat{\mathcal{S}}.
\end{equation}
A matrix $X_*$ minimizes \eqref{eq:coercive} if and only if $Y_* = A - BX_*C$ minimizes $f+ \eta$. Since $\widehat{\mathcal{S}}$ is the surjective image of $\mathcal{S}$ under $X \mapsto A-BXC$, \eqref{eq:minimizer} is equivalent to \eqref{eq:min}.
\end{proof}

We next demonstrate sufficient descent in Algorithm~\ref{alg:general}.
\begin{lemma}[Sufficient descent]\label{lem:sufficient_descent}
Let $Y_* = A - BX_* C$ minimize \eqref{eq:coercive} with $\Delta_* \in \partial f(Y_*)$ such that 
\begin{equation}\label{eq:D_gradient}
    \tr\bigl(\Delta_*^\tp B(X_*-X)C \bigr)\ge 0 \quad \text{for all } X \in \mathcal{S}.
\end{equation}
Then the iterate $(Y_k,\Delta_k)$ of Algorithm~\ref{alg:general} satisfies
\begin{align}
    \lVert Y_k-Y_* \rVert_\F^2 + \lVert \Delta_k-\Delta_* \rVert_\F^2 - \lVert Y_{k+1} -Y_* \rVert_\F^2 &- \lVert \Delta_{k+1} - \Delta_* \rVert_\F^2 \notag \\
&\ge \lVert Y_k - Y_{k+1} \rVert_\F^2  + \lVert \Delta_k - \Delta_{k+1} \rVert_\F^2. \label{eq:sufficient_descent}
\end{align}
\end{lemma}
\begin{proof} 
For each $k$, it follows from \eqref{eq:Y_step} and \eqref{eq:D_step} that $\mu \Delta_k \in \partial f(Y_k)$.
Since $\mu \Delta_k \in \partial f(Y_k)$ and $\mu \Delta_{k+1} \in \partial f(Y_{k+1})$,
\begin{equation}\label{eq:gradient_sp2}
\begin{split}
    \tr\bigl((\mu \Delta_k- \mu \Delta_{k+1})^\tp (Y_k-Y_{k+1})\bigr) &= \tr\bigl(\mu \Delta_k^\tp (Y_k-Y_{k+1})\bigr) - \tr\bigl(\mu \Delta_{k+1}^\tp (Y_k-Y_{k+1})\bigr) \\
&= \tr\bigl(\mu \Delta_k^\tp (Y_k-Y_{k+1})\bigr) + \tr\bigl(\mu \Delta_{k+1}^\tp (Y_{k+1}-Y_k)\bigr) \\
&\ge f(Y_k) - f(Y_{k+1}) + f(Y_{k+1}) - f(Y_k) = 0.
\end{split}
\end{equation}
By \eqref{eq:X_step}, $\lVert Y_{k+1} - A + BXC - \Delta_k \rVert_\F^2 - \lVert Y_{k+1} - A + BX_{k+1}C - \Delta_k \rVert_\F^2 \ge 0$ for all $X \in \mathcal{S}$. By
\begin{equation}\label{eq:factorization}
    \lVert a - c \rVert_\F^2 + \lVert b - c \rVert_\F^2 = 2 \tr\bigl((a-c)^\tp (a-b) \bigr) - \lVert a - b \rVert_\F^2,
\end{equation}
we get $2\tr\bigl((Y_k-A+BXC-\Delta_k)^\tp B(X-X_{k+1})C\bigr) \ge \lVert B(X-X_{k+1})C \rVert_\F^2$.
Plugging in \eqref{eq:D_step} gives
\begin{equation}\label{eq:k_inequality}
    \tr\bigl((Y_k - Y_{k+1} - \Delta_{k+1})^\tp B(X-X_{k+1})C\bigr) \ge 0.
\end{equation}
Taking $X = X_*$ and adding to \eqref{eq:D_gradient}, we get
$\tr\bigl((Y_k - Y_{k+1} + \Delta_* - \Delta_{k+1})^\tp (Y_{k+1} - Y_* + \Delta_{k+1} - \Delta_k) \bigr) \ge 0$.
By \eqref{eq:D_step} and \eqref{eq:gradient_sp2},
\begin{align*}
\tr \bigl((\Delta_k-\Delta_{k+1})^\tp (\Delta_{k+1}-\Delta_*) \bigr) &+ \tr\bigl( (Y_{k+1}-Y_*)^\tp (Y_k-Y_{k+1}) \bigr)\\
&\ge \tr\bigl((\Delta_k-\Delta_{k+1})^\tp(Y_k-Y_{k+1})\bigr) \ge 0.
\end{align*}
By \eqref{eq:factorization} again, 
\begin{align*}
    \lVert Y_k-Y_* \rVert_\F^2 + \lVert \Delta_k-\Delta_* \rVert_\F^2 &- \lVert Y_{k+1} -Y_* \rVert_\F^2 - \lVert \Delta_{k+1} - \Delta_* \rVert_\F^2 \\
    &=\begin{multlined}[t] \lVert Y_k - Y_{k+1} \rVert_\F^2  + \lVert \Delta_k - \Delta_{k+1} \rVert_\F^2 \\
+ 2\tr\bigl((Y_k - Y_{k+1})^\tp (Y_{k+1}-Y_*) +(\Delta_k - \Delta_{k+1})^\tp (\Delta_{k+1}-\Delta_*)\bigr) \end{multlined} \\
    &\ge \lVert Y_k - Y_{k+1} \rVert_\F^2  + \lVert \Delta_k - \Delta_{k+1} \rVert_\F^2. \qedhere
\end{align*}
\end{proof}

We now arrive at the required convergence result.
\begin{theorem}[Global, provable, and sublinear convergence]\label{thm:conv}
Let $f: \R^{m \times n} \to \R$ be subdifferentiable and attain its minimum on $\widehat{\mathcal{S}} \coloneqq \{A - BXC \in \R^{m \times n} : X \in \mathcal{S}\}$. For any $\mu > 0$, the iterates $(Y_k, \Delta_k)$ in Algorithm~\ref{alg:general} will converge to $(Y_*, \Delta_*)$ such that $Y_* = A - BX_*C$ for some $X_* \in \R^{s \times t}$ that minimizes \eqref{eq:coercive} and $\Delta_* \in \partial f(Y_*)$ that satisfies \eqref{eq:D_gradient}. The convergence is sublinear as
\begin{equation}\label{eq:sublin}
\frac{\lVert Y_{k+1} -Y_* \rVert_\F^2 + \lVert \Delta_{k+1} - \Delta_* \rVert_\F^2}{\lVert Y_k -Y_* \rVert_\F^2 + \lVert \Delta_k - \Delta_* \rVert_\F^2} \le 1, \quad k =0, 1, 2,\dots.
\end{equation}
If $B$ has full column rank and $C$ has full row rank, then $X_*$ is unique and the iterates $X_k$ in Algorithm~\ref{alg:general} will converge to $X_*$.
\end{theorem}
\begin{proof}
By Lemma~\ref{lem:sufficient_descent}, the sequence $(\lVert Y_k - Y_* \rVert_\F + \lVert \Delta_k - \Delta_* \rVert_\F)_{k=0}^\infty$ is monotone decreasing and bounded below, thus convergent. It remains to show that the limit is zero. By passing through a subsequence if necessary, we may assume that $(Y_k)_{k=0}^\infty$ and $(\Delta_k)_{k=0}^\infty$ are convergent with limits $\widehat{Y}$ and  $\widehat{\Delta}$ respectively. Taking limit in \eqref{eq:k_inequality}, we get $\tr\bigl(\widehat{\Delta}^\tp(\widehat{Y} - A + BXC) \bigr) \ge 0$ for all $X \in \mathcal{S}$.
By Lemma~\ref{lem:opt}, $\widehat{Y}$ is a global minimizer. Let $Y_* = \widehat{Y}$ and $\Delta_* = \widehat{\Delta}$. Since $(\lVert Y_k - Y_* \rVert_\F + \lVert \Delta_k - \Delta_* \rVert_\F)_{k=0}^\infty$ is convergent, we must have $(Y_k,\Delta_k) \to (Y_*, \Delta_*)$. We rearrange \eqref{eq:sufficient_descent} to get \eqref{eq:sublin}.

If $B$ has full column rank and $C$ has full row rank, then $X \mapsto BXC$ is injective, so $X_*$ is unique. Taking limit in \eqref{eq:D_step} gives
$\lim_{k\to \infty} \lVert B(X_* - X_k)C \rVert_\F = 0$ and thus $\lim_{k\to \infty} \lVert X_* - X_k\rVert_\F = 0$.
\end{proof}

\subsection{Algorithm for generalized nearness in Schatten norms}\label{sec:Schatten}

Any norm $\lVert \, \cdot \, \rVert$ is subdifferentiable. To apply Theorem~\ref{thm:conv} to $f = \lVert \, \cdot \, \rVert$, we just need to guarantee that its minimum is attained on $\widehat{\mathcal{S}}$. We will show that this holds whenever $\lVert \, \cdot \, \rVert$ is orthogonally invariant and $\mathcal{S}$ is defined by the constraints~\ref{it:rank}--\ref{it:pos}, with one small exception --- the positive semidefinite subcase in \ref{it:pos}. We refer readers to the discussion following \cite[Proposition~4.1]{li22} for a counterexample for this subcase. We will also show that Theorem~\ref{thm:conv} applies to the generalized nearness problems considered in Sections~\ref{sec:Kronecker} and  \ref{sec:partial-trace}

\begin{proposition}[Minima attainment]\label{prop:attain}
Let $ A$, $B$, $B_1$, $B_2$, $C$, $C_1$, $C_2$ be matrices of appropriate dimensions. Let $\lVert\,\cdot\,\rVert$ be any orthogonally invariant norm.
\begin{enumerate}[\upshape (a)]
\item The function defined by $f(X) = \lVert A - BXC \rVert$ attains its minimum on any $\mathcal{S}$ defined by constraints~\ref{it:rank}--\ref{it:pos} except for the positive semidefinite case in \ref{it:pos}. 

\item The function defined by $f(X) = \lVert  A- (B_1 \otimes B_2)X(C_1 \otimes C_2)  \rVert$ attains its minimum on
\begin{align*}
\mathcal{S}_1 &= \{X \in \R^{p_1p_2 \times q_1q_2}: \rank_\otimes(X) \le r\},\\
\mathcal{S}_2 &= \{X \in \R^{p^2 \times p^2}: \tr_\Par(X(W \otimes I)) = \lambda W \text{ for some } W \in \R^{p \times p}\}.
\end{align*}
\end{enumerate}
\end{proposition}
\begin{proof} 
Let $\mathbb{V} = \{X \in \R^{p \times q} : BXC = 0\}$ and $\pi$ be the orthogonal projection onto $\mathbb{V}^\perp$. Since $\lVert A - BXC \rVert = \lVert A - B\pi(X)C \rVert$, the function $f$ attains a minimum on $\mathcal{S}$ if and only if it attains a minimum on $\pi(\mathcal{S})$; and the latter happens if the sublevel set $\mathcal{E}_\alpha \coloneqq \{X \in \pi(\mathcal{S}) : \lVert A - BXC \rVert \le \alpha\}$ is compact for some $\alpha \ge 0$. Since the restriction $f\bigr\rvert_{\mathbb{V}^\perp}$ satisfies $f\bigr\rvert_{\mathbb{V}^\perp}(X) \to \infty$ whenever $\lVert X \rVert \to \infty$, the sublevel set $\mathcal{E}_\alpha$ is bounded. Thus it suffices to check that $\pi(\mathcal{S})$ is closed.
\begin{enumerate}[\normalfont(i)]
\item In \ref{it:norm} and the correlation, stochastic, and doubly stochastic subcases in \ref{it:pos}, the set $\mathcal{S}$ is compact, thus $\pi(\mathcal{S})$ is closed as $\pi$ is continuous. 
\item In \ref{it:prod}, \ref{it:sym}, \ref{it:spec}, and the prescribed eigenvector subcase in \ref{it:eig}, the set $\mathcal{S}$ is an affine subspace, thus so is $\pi(\mathcal{S})$. Every affine subspace is closed.
\item In the nonnegative subcase in \ref{it:pos}, the set $\mathcal{S}$ is a polyhedral cone, thus so is $\pi(\mathcal{S})$. Every polyhedral cone is closed since it is a finite intersection of halfspaces.
\end{enumerate}
For the rank constrained case \ref{it:rank}, $\mathcal{S} = \{X \in \R^{p \times q} : \rank(X) \le r\}$ and $\widehat{\mathcal{S}} = \{Y \in \R^{s \times t} : \rank(Y) \le r\}$. Let  $A_{11}$, $A_{12}$, $A_{21}$, $A_{22}$ be as in \eqref{eq:equiv1} and $\widehat{f} : \R^{s \times t} \to \R$ be defined by
\[
\widehat{f}(Y) = \biggl \lVert \begin{bmatrix}
    A_{11}-Y &A_{12}\\
A_{21} & A_{22}
\end{bmatrix} \biggr\rVert.
\]
Then, by \eqref{eq:equiv2}, 
\[
    \inf_{X \in \mathcal{S}} \; f(X) = \inf_{Y \in \widehat{\mathcal{S}}} \; \widehat{f}(Y).
\]
Since $\widehat{f}(Y_k) \to \infty$ whenever $\lVert Y_k \rVert \to \infty$ and $\widehat{\mathcal{S}}$ is closed, $\widehat{f}(Y)$ admits a minimum on $\widehat{\mathcal{S}}$. This proves case \ref{it:rank}. For the Kronecker rank constrained case, we use $\mathcal{S}_1$ in place of $\mathcal{S}$. By the singular value decompositions of $B_1 \otimes B_2$ and $C_1 \otimes C_2$, we obtain a relation similar to \eqref{eq:equiv2}. The rest of the argument follows verbatim.

The prescribed eigenvalue subcase in \ref{it:eig} reduces to the rank constrained case above, as $\lambda \in \lambda(X)$ if and only if $\rank(X -\lambda I)\le p-1$, so
\[
\min_{\lambda \in \lambda(X)} \lVert A - BXC \rVert = \min_{\rank(\widehat{X})\le p-1}\lVert (A -\lambda BC) - B\widehat{X}C \rVert.
\]
Likewise, the prescribed partial trace case reduces to the Kronecker rank constrained case by \eqref{eq:partial-eig}, giving us
\[
\min_{\rank_\otimes(\widehat{X})\le p^2-1} \lVert (A - \lambda (B_1 \otimes B_2)\mathcal{R}^{-1}(I)(C_1 \otimes C_2)) - (B_1 \otimes B_2)\widehat{X}(C_1 \otimes C_2) \rVert. \qedhere
\]
\end{proof}

Any orthogonally invariant function $f:\R^{m \times n} \to \R$ is a function of singular values. In fact, it must take the form $f = h \circ \sigma$ where  $h: \R^{\min(m,n)} \to \R$ is invariant under sign changes and permutations of components and $\sigma$ is as in \eqref{eq:sing} \cite{Lewis_UI}. In particular, if $h = \lVert \, \cdot \, \rVert$ is a norm on $\R^{\min(m,n)}$, then $h \circ \sigma$ is an orthogonally invariant norm on $\R^{m \times n}$. We also have the following formulas for subdifferentials:
\begin{align}
    \partial f(Y) &= \bigl\{U\diag(g)V^\tp : Y = U\Sigma V^\tp,\, g \in \partial h\bigl(\sigma(Y)\bigr) \bigr\}, \label{eq:uinorm_subdifferential} \\
    \partial\lVert v \rVert &= \{g \in \R^{\min(m,n)} : g^\tp v = \lVert v \rVert, \, \lVert g \rVert^* \le 1 \}. \label{eq:norm_sub}
\end{align} 
With these, we may explicitly solve the $Y_k$-update step \eqref{eq:Y_step} in Algorithm~\ref{alg:general} for any $f = \lVert \, \cdot \, \rVert_{\st, p}$ with $p \in [1,\infty]$.
\begin{lemma}[$Y_k$-update for nuclear norm]\label{lem:nuc}
Let $\mu>0$ and $M \in \R^{m \times n}$ with  singular value decomposition $M = U\Sigma V^\tp$. Let $\sigma_* \in \R^{\min(m,n)}$ be given by  
\[\sigma_i^* = 
\begin{dcases}
\sigma_i(M)-1/\mu  & \text{if } \sigma_i(M)>1/\mu,\\
0 & \text{if } \sigma_i(M)\le 1/\mu,
\end{dcases}
\]
for $i = 1,\dots, \min(m,n)$. Then
\[
Y_* \coloneqq U \diag(\sigma_*) V^\tp \in \argmin  \bigl\{ \lVert Y \rVert_{\st,1} + \tfrac{\mu}{2} \lVert Y - M \rVert_\F^2 : Y \in \R^{m \times n}\bigr\}.
\]
\end{lemma}
\begin{proof}
Choose $\ell\in \{1,\dots,\min(m,n)\}$ such that
\[\begin{dcases}
\sigma_i(M) \ge 1/\mu  & \text{if } i \le \ell,\\
\sigma_i(M) < 1/\mu & \text{if } i > \ell.
\end{dcases}\]
Let $g \in \R^{\min (m, n)}$ be defined by
\[
g_{i} = \begin{dcases}
    1  & \text{if } i \le \ell,\\
    \mu \sigma_i(M) & \text{if } i > \ell.
    \end{dcases}
\]
Since $\lVert g \lVert_\infty = 1$ and 
\[
    g^\tp \sigma_* = \sum_{i=1}^\ell \bigl(\sigma_i(M)-1/\mu \bigr) = \lVert \sigma_* \rVert_1,
\]
we obtain $g\in \partial \lVert \sigma_* \rVert_1$ by \eqref{eq:norm_sub}. It follows that $U\diag(g)V^\tp \in \partial \lVert Y_* \rVert_{\st,1}$ by \eqref{eq:uinorm_subdifferential}. So $Y_*$ is a required minimizer as
\[
U\diag(g)V^\tp + \mu (Y_* -M) = 0 \in \partial \bigl( \lVert Y_* \rVert_{\st,1} + \tfrac{\mu}{2} \lVert Y_* - M \rVert_\F^2\bigr). \qedhere
\]
\end{proof}

\begin{lemma}[$Y_k$-update for spectral norm]\label{lem:spec}
Let $\mu>0$ and $M \in \R^{m \times n}$ with  singular value decomposition $M = U\Sigma V^\tp$. Let
\begin{equation}\label{eq:pick-a}
    \ell = \max\biggl\{ i \in \{1,\dots, \min(m,n)\} : \sigma_i(M) \ge \frac{1}{i}\biggl(\sum_{j=1}^i \sigma_j(M) -\frac{1}{\mu}\biggr) \biggr\}
\end{equation}
and $\sigma_* \in \R^{\min(m,n)}$ be given by 
\[\sigma_i^* = 
\begin{dcases}
\frac{1}{\ell} \biggl( \sum_{j=1}^\ell \sigma_j(M) -\frac{1}{\mu} \biggr) & \text{if }   i \le \ell,\\
\sigma_i(M) &\text{if } i > \ell,
\end{dcases}
\]
for $i = 1,\dots, \min(m,n)$.  Then
\[
Y_* \coloneqq U \diag(\sigma_*) V^\tp \in \argmin  \bigl\{ \lVert Y \rVert_{\st,\infty} + \tfrac{\mu}{2} \lVert Y - M \rVert_\F^2 : Y \in \R^{m \times n}\bigr\}.
\]
\end{lemma}
\begin{proof}
Let $g \in \R^{\min(m, n)}$ be given by 
\[
g_{i} = \begin{dcases}
    \frac{1}{\ell} + \mu \sigma_i(M) - \frac{1}{\ell}\sum_{j=1}^\ell \mu \sigma_j(M)  &\text{if } i \le \ell,\\
    0 &\text{if } i > \ell.
    \end{dcases}
\]
By \eqref{eq:pick-a}, 
\[
\mu \sigma_i(M) \ge \mu\sigma_\ell(M) \ge \frac{1}{\ell}\biggl(\sum_{j=1}^\ell \mu\sigma_j(M) -1\biggr)
\]
for all $i = 1,\dots, \ell$. So $g$ is a nonnegative vector. If $\ell <\min(m,n)$, then the maximality of $\ell$ implies that 
\[
\frac{1}{\ell}\biggl(\sum_{i=1}^\ell \sigma_i(M) -\frac{1}{\mu}\biggr) \ge \sigma_{\ell+1}.
\]
The coordinates of $\sigma_*$ are thus in nondecreasing order. By \cite[Example~1]{UIsubdifferential},
\[\partial \lVert \sigma_* \rVert_\infty = \conv \{e_i : i = 1,\dots, \ell \}.\]
Since
\[
    \sum_{i=1}^\ell \biggl( \frac{1}{\ell} + \mu \sigma_i(M) - \frac{1}{\ell}\sum_{j=1}^\ell \mu \sigma_j(M) \biggr) = 1,
\]
we obtain $g \in \partial \lVert \sigma_*\rVert_\infty$. It follows that $U\diag(g)V^\tp \in \partial \lVert Y_* \rVert_{\st,\infty}$ by \eqref{eq:uinorm_subdifferential}. So $Y_*$ is a required minimizer as
\[
    U\diag(g)V^\tp + \mu(Y_*-M) = 0\in \partial \bigl( \lVert Y_* \rVert_{\st,\infty} + \tfrac{\mu}{2} \lVert Y_* - M \rVert_\F^2\bigr). \qedhere
\]
\end{proof}
\begin{lemma}[$Y_k$-update for Schatten norm]\label{lem:sch}
Let $p \in (1,\infty)$, $\mu>0$, and $M \in \R^{m \times n}$ with  singular value decomposition $M = U\Sigma V^\tp$. Let $\sigma_* \in \R^{\min(m,n)}$ be given by the unique root in $[0,\infty)$ of
\[
p\sigma_i^{\ast \smash[b]{p-1}} + \mu (\sigma_i^* - \sigma_i(M)) = 0,
\]
for $ i = 1, \dots, \min(m,n)$. Then
\[
Y_* \coloneqq U \diag(\sigma_*) V^\tp \in \argmin  \bigl\{ \lVert Y\rVert_{\st,p}^p + \tfrac{\mu}{2} \lVert Y - M \rVert_\F^2 : Y \in \R^{m \times n}\bigr\}.
\]
\end{lemma}
\begin{proof}
Since $\varphi(x) \coloneqq px^{p-1}+\mu(x-\sigma_i(M))$ is continuous and strictly increases to $\infty$ on $[0,\infty)$ and evaluates to $-\mu\sigma_i(M) \leq 0$ at $x=0$, it has a unique zero in $[0,\infty)$ for each $i = 1, \dots, \min(m,n)$. By \eqref{eq:uinorm_subdifferential}, for any $Y \in \R^{m \times n}$,
\[
\partial \lVert Y \rVert_{\st,p}^p = \bigl\{U\diag(p\sigma_1(Y)^{p-1}, \dots, p\sigma_{\min(m,n)}(Y)^{p-1})V^\tp : Y = U\Sigma V^\tp\bigr\}.
\] 
So $Y_*$ is a required minimizer as $0 \in \partial  \bigl( \lVert Y_* \rVert_{\st,p}^p + \tfrac{\mu}{2} \lVert Y_* - M \rVert_\F^2\bigr)$.
\end{proof}

With the solutions of the $Y_k$-update step \eqref{eq:Y_step} in Lemmas~\ref{lem:nuc}, \ref{lem:spec}, \ref{lem:sch}, Algorithm~\ref{alg:general} now takes the explicit form below. Convergence follows from Theorem~\ref{thm:conv} and Proposition~\ref{prop:attain}.

\begin{algorithm}
\caption{Algorithm for generalized nearness in Schatten norms}\label{alg:Schatten}
\begin{algorithmic}[1]
\Require $ A \in \R^{m \times n}$, $B \in \R^{m \times s}$, $C \in \R^{t \times n} $, $\mu > 0$, $p \in [1,\infty]$;
\State initialize $X_0$, $Y_0$, $\Delta_0$, and $k = 0$;
\State solve $X_k$-update step \eqref{eq:X_step} for $X_{k+1}$ via methods in \cite{li22}, Propositions~\ref{prop:Kronecker-vecb} or \ref{prop:partial-trace}; \label{line:X_step2}
\State compute singular value decomposition of $M \coloneqq A + \Delta_k - BX_{k+1}C = U \Sigma V^\tp$;
\If{$p = 1$}
    \State compute $\sigma_* \in \R^{\min (m,n)}$ with
    \[
    \sigma_i^* = 
        \begin{dcases}
        \sigma_i (M)-1/\mu  & \text{if } \sigma_i (M) >1/\mu,\\
        0 & \text{if } \sigma_i(M) \le 1/\mu;
        \end{dcases} 
    \]
\ElsIf{$p = \infty$}
    \State compute 
    \[\ell = \max\biggl\{ i \in \{1,\dots, \min(m,n)\} :  \sigma_i(M) \ge \frac{1}{i}\biggl(\sum_{j=1}^i \sigma_j(M) -\frac{1}{\mu}\biggr) \biggr\};\]
    \State compute $\sigma_* \in \R^{\min (m,n)}$ with
    \[
    \sigma_i^* = 
    \begin{dcases}
    \frac{1}{\ell}\biggl(\sum_{j=1}^\ell \sigma_j(M) -\frac{1}{\mu} \biggr)  & \text{if } i \le \ell,\\
    \sigma_i(M) &\text{if } i > \ell;
    \end{dcases}
    \]
\Else 
    \For {$k = 1,\dots, \min(m,n)$}
        \State solve $pz^{p-1} + \mu (z - \sigma_k(M)) = 0$ for $\sigma_k^*$;
    \EndFor
\EndIf
\State compute $Y_{k+1} = U \diag(\sigma_*) V^\tp$;
\State compute $\Delta_{k+1} = \Delta_k - Y_{k+1} + A - BX_{k+1}C$;
\State $k = k+1$ and go to line~\eqref{line:X_step2}.
\end{algorithmic}
\end{algorithm}

A departure from general purpose optimization algorithms is that Algorithm~\ref{alg:Schatten} does not involve numerical differentiation, symbolic differentiation, or automatic differentiation. This also applies to the methods referenced in Step~\ref{line:X_step2}, as those in \cite{li22} are notable for being zeroth order (no derivatives), as are those in Section~\ref{sec:closed-form}.

\section{Numerical experiments}\label{sec:experiments}

We present three sets of numerical experiments to demonstrate the advantages of Algorithm~\ref{alg:Schatten} over standard convex optimization methods and to  illustrate its utility in real-world applications. All codes have been made available at \url{https://github.com/thomasw15/near2}. All experiments are done in Python; the version of \texttt{CVX} below refers to \texttt{CVXPY} \cite{cvx}.

\subsection{Matrix recovery}

Here we will see the performance of Algorithm~\ref{alg:Schatten} on
\[
\min_{X \in \mathcal{S} } \; \lVert A - BXC \rVert_{\st,p}
\] 
for $p=1$, $\frac32$, and $\infty$ in the (i) unconstrained, (ii) two-sided product, (iii) prescribed eigenvalue, and (iv) rank constrained generalized nearness problem. The cases (i) and (ii)  are convex, and we will compare our results against the popular package \texttt{CVX}  with its \texttt{SCS} solver \cite{SCS} for nuclear norm ($p=1$) and spectral norm ($p=\infty$). The other two cases (iii) and (iv) are nonconvex, so \texttt{CVX} does not work on these, and we simply record our results. What might surprise the uninitiated is that there is also no conceivable way to solve cases (i) and (ii) in the Schatten $\frac32$-norm with \texttt{CVX}, convexity of the problems notwithstanding. General purpose packages like \texttt{CVX} are designed for problems that can be transformed into a handful of standard convex programs (e.g., LP, QP, SOCP, SDP, GP) and are all but useless for convex objective functions and constraints that cannot be put in these forms. Indeed, we picked the unusual value $p = \frac32$ to make this point.

This experiment is designed for maximal control so that we know the true solutions a priori and may compare forward errors. We ``work backwards'' to set up our data:  generate random matrices $B, C, X \in \R^{n \times n}$ where $X$ satisfies the desired constraint, and then set $A \coloneqq BXC$. This simulates the common scenario in which one seeks to fit $A = BXC$ in the presence of errors by minimizing $\lVert A - BXC \rVert_{\st,p}$ subject to some known structure of $X$. Given that we know the true solution $X$ explicitly, we may measure the relative forward error $\lVert \widehat{X} - X \rVert_\F /\lVert X \rVert_\F$, where $\widehat{X}$ is the solution computed from $A,B,C$ using either Algorithm~\ref{alg:Schatten} or \texttt{CVX}.

\subsubsection{Speed} 

For each dimension $n = 2^k$, $k = 1,\dots, 10$, we quantify the speed of Algorithm~\ref{alg:Schatten} as its average run times for forward error to reach $10^{-8}$ over ten experiments. For the cases (i) and (ii), and when the norm is nuclear or spectral, we compare our results with the \texttt{SCS} solver in \texttt{CVX}.  Figure~\ref{fig:speed} shows that Algorithm~\ref{alg:Schatten} converges significantly faster than \texttt{CVX} in all cases. In fact, \texttt{CVX} fails to converge for larger dimension $n$ within a reasonable time limit. For example, when $n = 2^7$, \texttt{CVX} took about two hours to solve the two-sided product constrained problem (ii) when Algorithm~\ref{alg:Schatten} took mere hundredth of a second. For $n = 2^8$ and beyond, \texttt{CVX} did not converge within $24$ hours.

\begin{figure}[htb]
\minipage{0.49\textwidth}
    \includegraphics[trim={0.6em 1.4ex .5em 0.8ex}, clip, width=\textwidth]{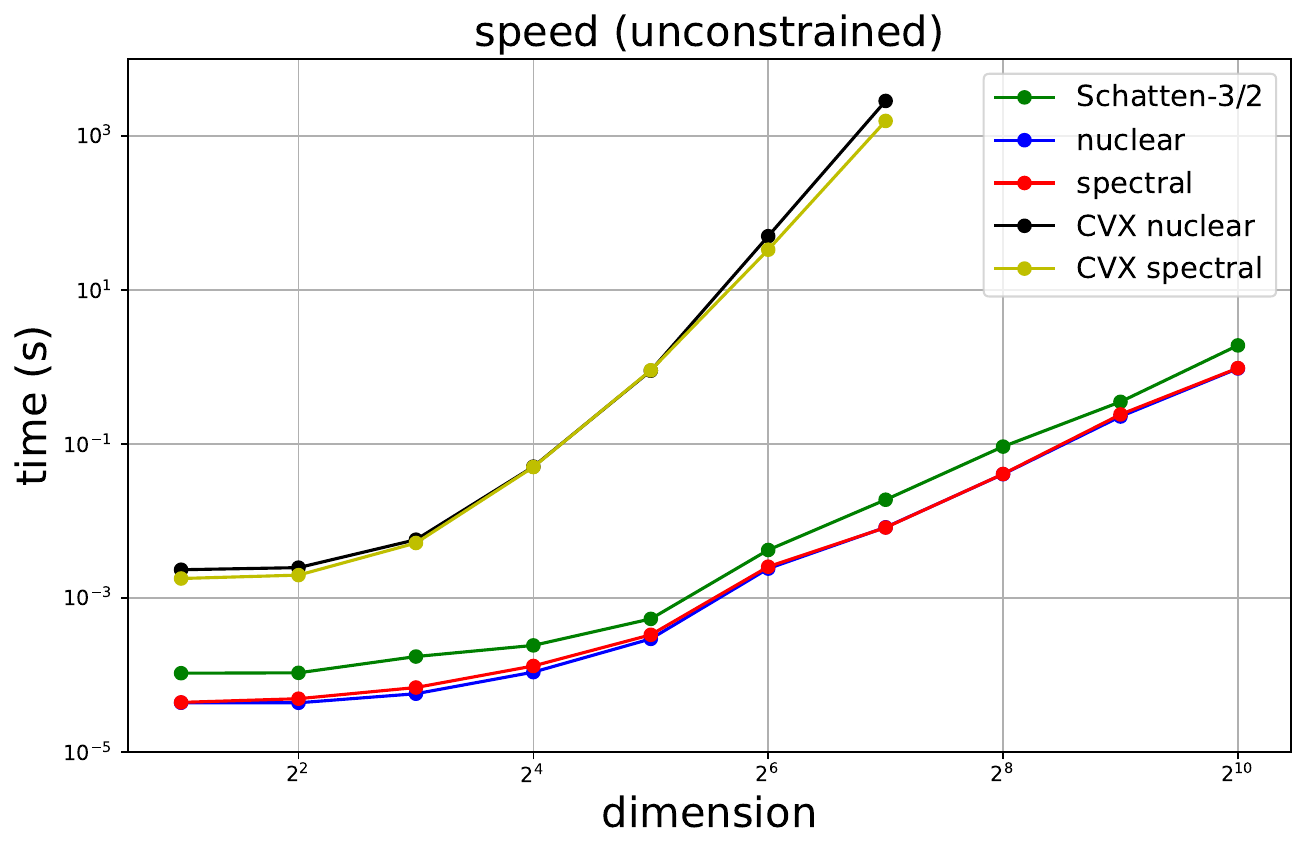}
\endminipage\hfill
\minipage{0.49\textwidth}
    \includegraphics[trim={0.6em 1.4ex .5em 0.8ex}, clip, width=\textwidth] {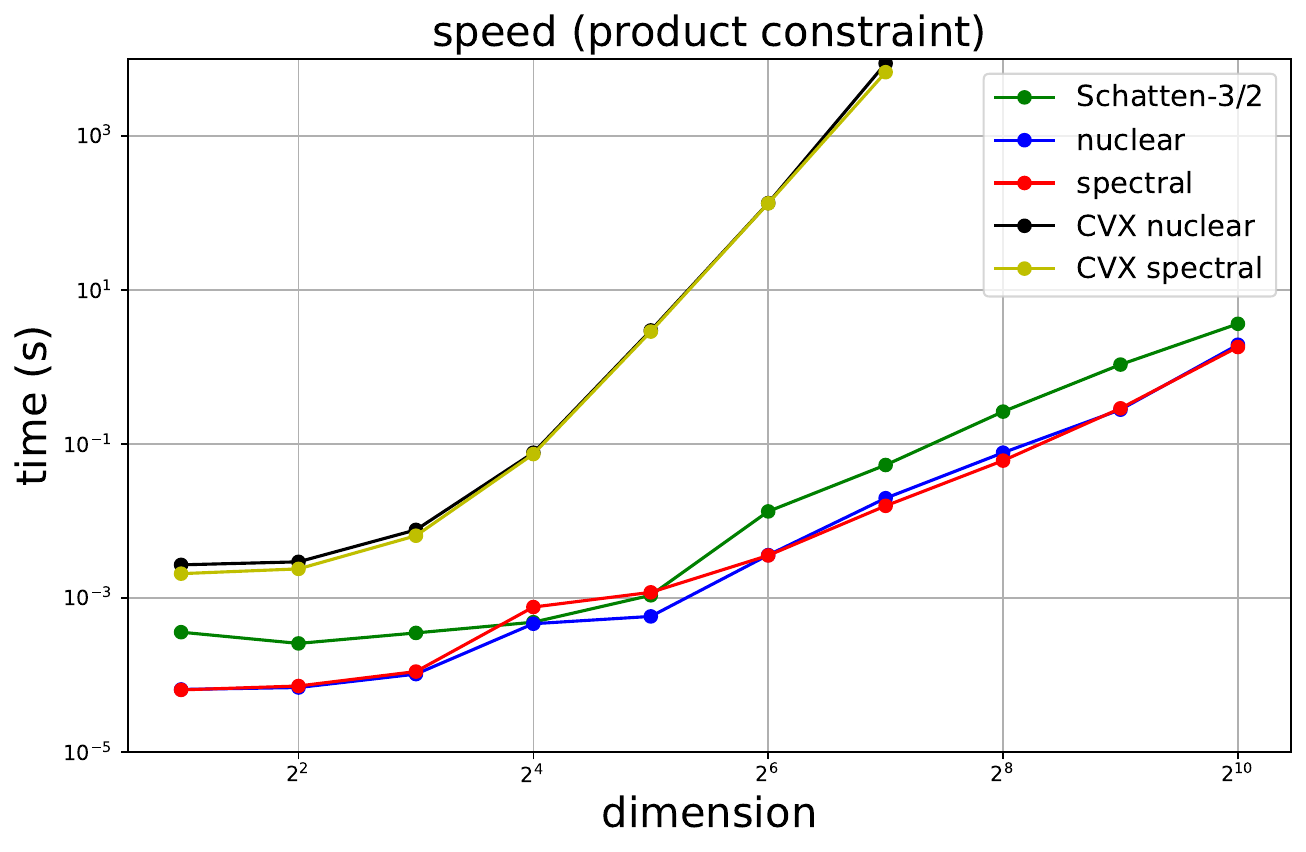}
\endminipage\hfill
\minipage{0.49\textwidth}
    \includegraphics[trim={0.6em 1.4ex .5em 0.8ex}, clip, width=\textwidth]{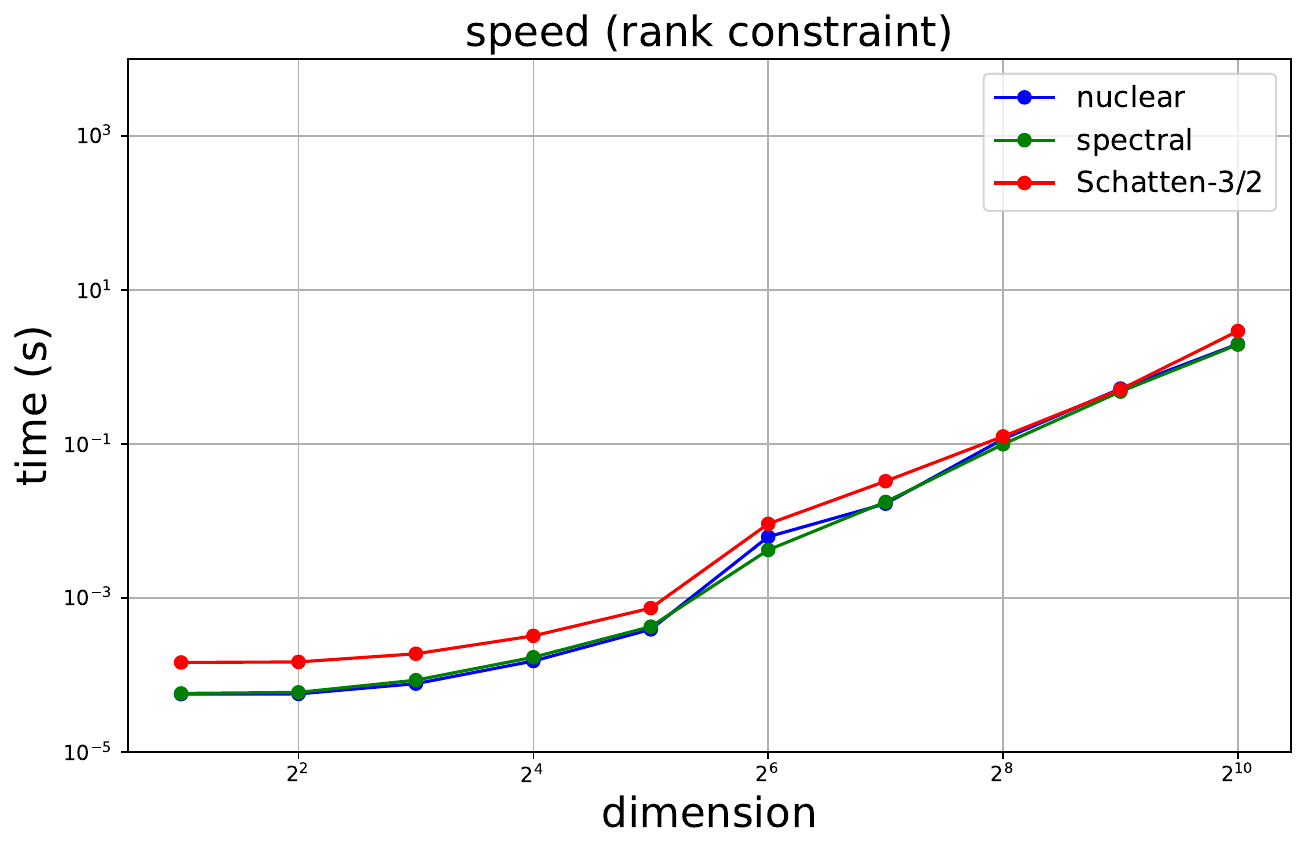}
\endminipage\hfill
\minipage{0.49\textwidth}
    \includegraphics[trim={0.6em 1.4ex .5em 0.8ex}, clip, width=\textwidth]{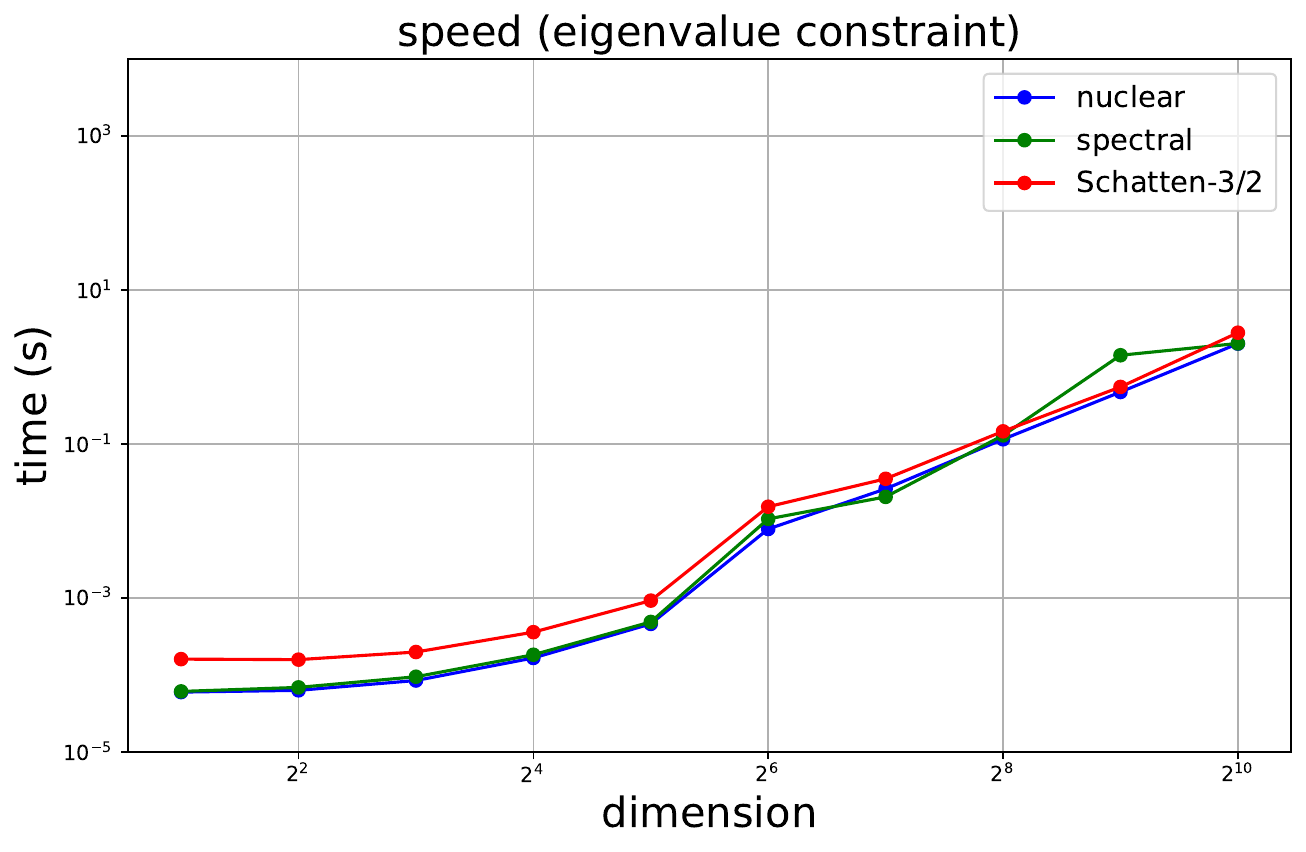}
\endminipage\hfill
\caption{Speed of Algorithm~\ref{alg:Schatten}. \texttt{CVX} comparison is unavailable for the bottom two problems, which are nonconvex, or the Schatten $\frac32$-norm, which is nonstandard.}
\label{fig:speed}
\end{figure}  
\FloatBarrier

\subsubsection{Accuracy}

We compute the relative forward errors of Algorithm~\ref{alg:Schatten} and compare them with \texttt{CVX}'s in the convex cases. We set $n = 32$, as \texttt{CVX} fails to converge in a reasonable amount of time for larger values of $n$.  For the two convex cases (i) and (ii), we set \texttt{CVX} to its maximum allowed precision, \texttt{epsrel} $=$ \texttt{epsabs} $= 10^{-16}$.  We record all forward errors over $20$ trials.

Figure~\ref{fig:accuracy} shows that Algorithm~\ref{alg:Schatten}'s forward errors consistently hover around $10^{-15}$, the level of machine precision, irrespective of whether the problem is convex. For the two convex problems (i) and (ii), when we also have forward errors from \texttt{CVX}, it is clear that  Algorithm~\ref{alg:Schatten}'s accuracy far exceeds that of \texttt{CVX}, which fluctuates wildly from $10^{-7}$ to $10^3$ in case (i). Accuracy of \texttt{CVX}, while still lagging behind that of Algorithm~\ref{alg:Schatten}, shows an improvement in case (ii) because of the dimension-lowering effect of the constraint $FXG = H$.

\begin{figure}[htb]
    \minipage{0.49\textwidth}
        \includegraphics[trim={0.6em 1.4ex .5em 0.8ex}, clip, width=\textwidth]{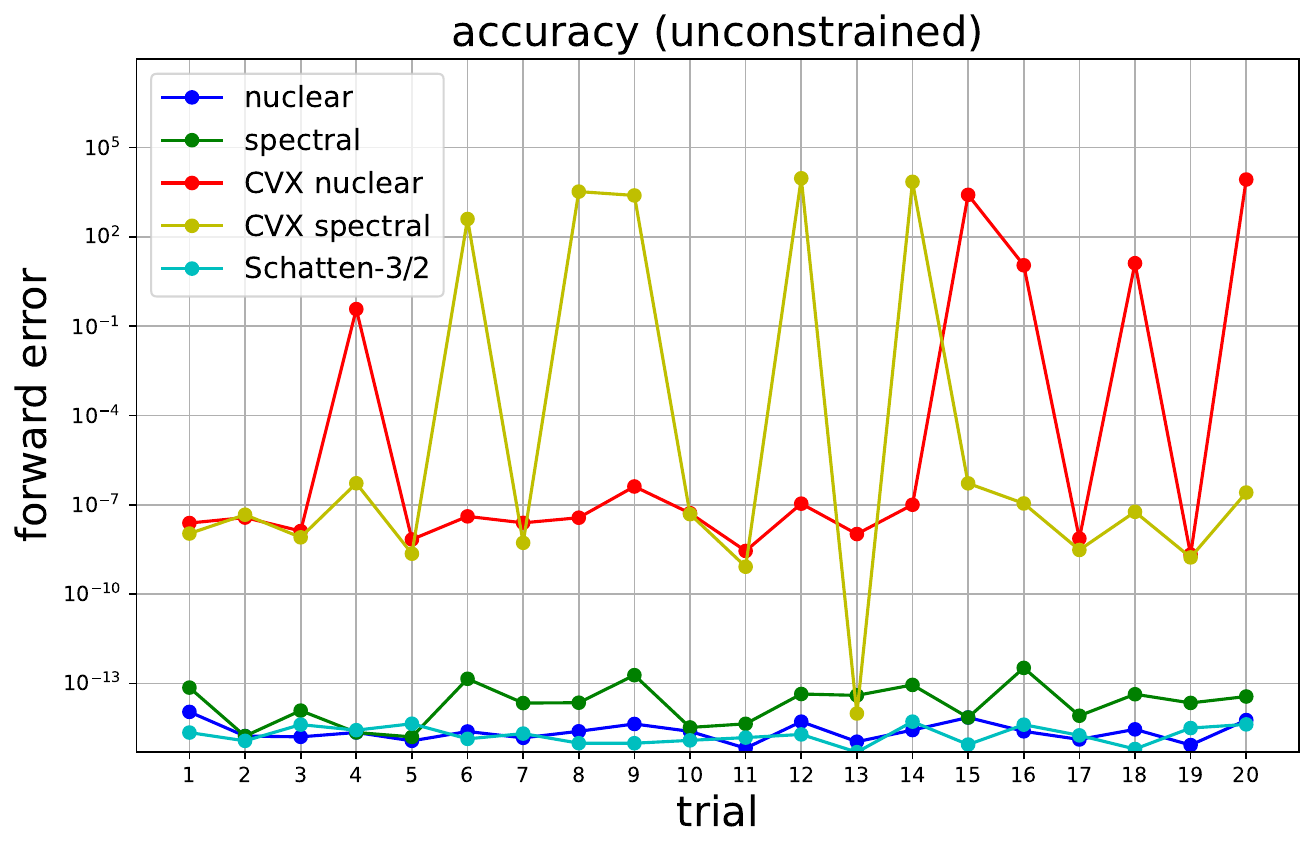}
    \endminipage\hfill
    \minipage{0.49\textwidth}
        \includegraphics[trim={0.6em 1.4ex .5em 0.8ex}, clip, width=\textwidth]{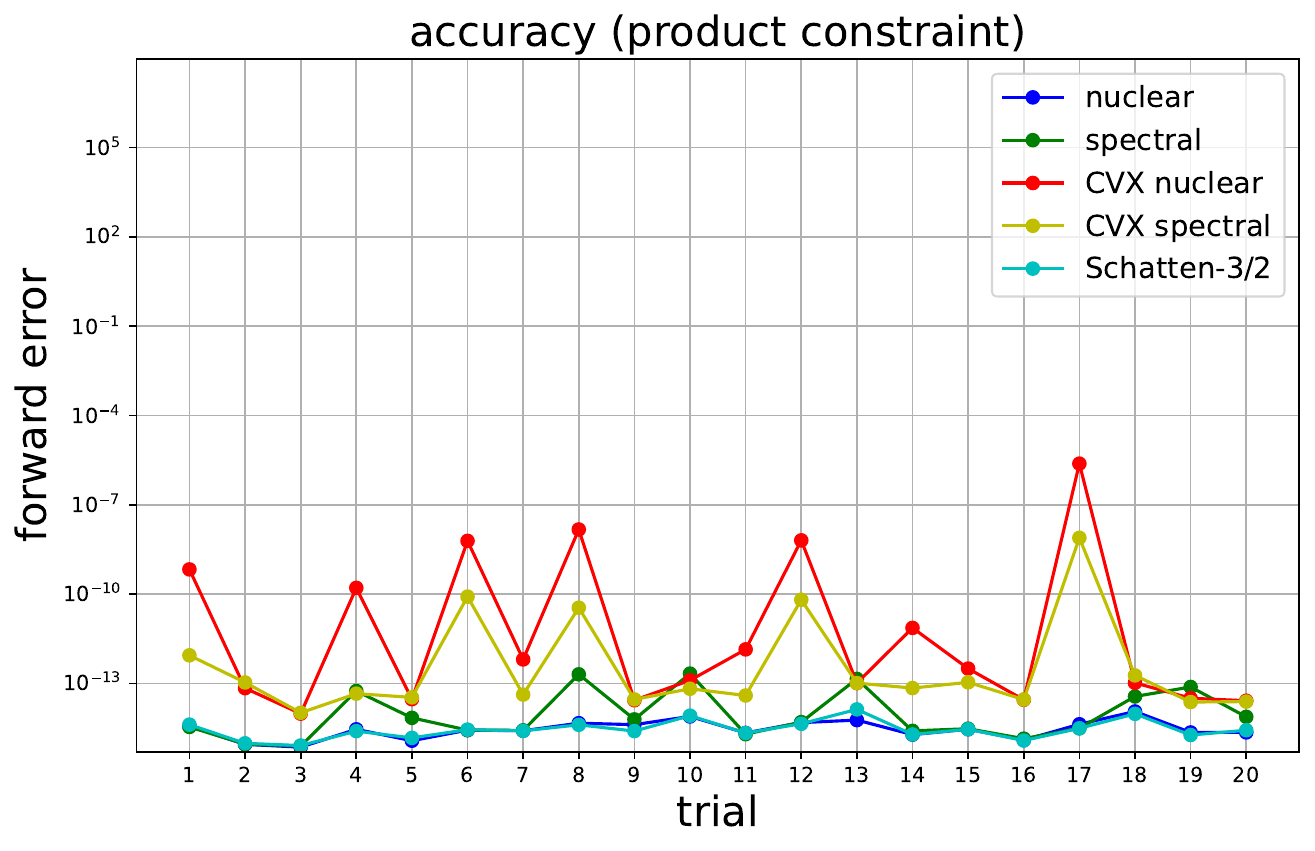}
    \endminipage\hfill
    \minipage{0.49\textwidth}
        \includegraphics[trim={0.6em 1.4ex .5em 0.8ex}, clip, width=\textwidth]{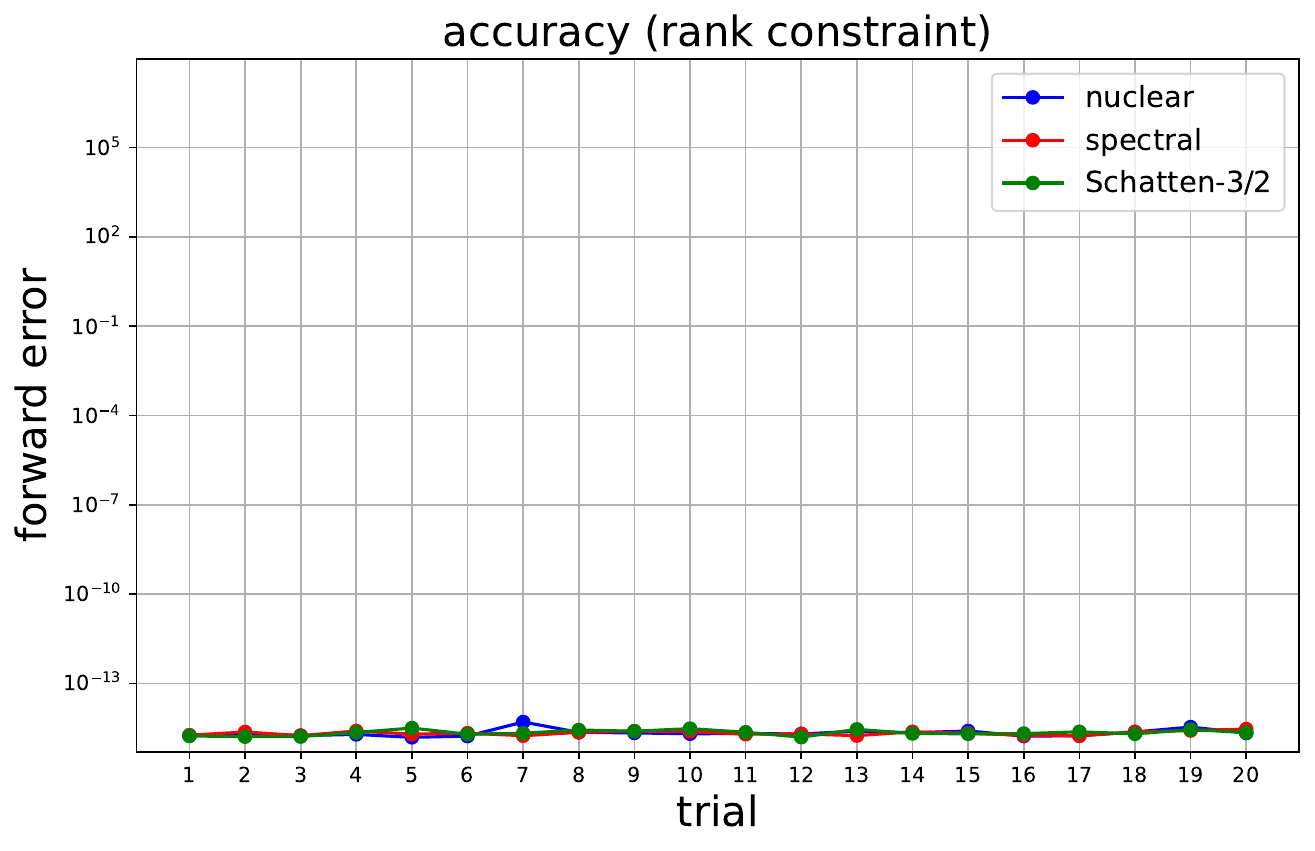}
    \endminipage\hfill
    \minipage{0.49\textwidth}
        \includegraphics[trim={0.6em 1.4ex .5em 0.8ex}, clip, width=\textwidth]{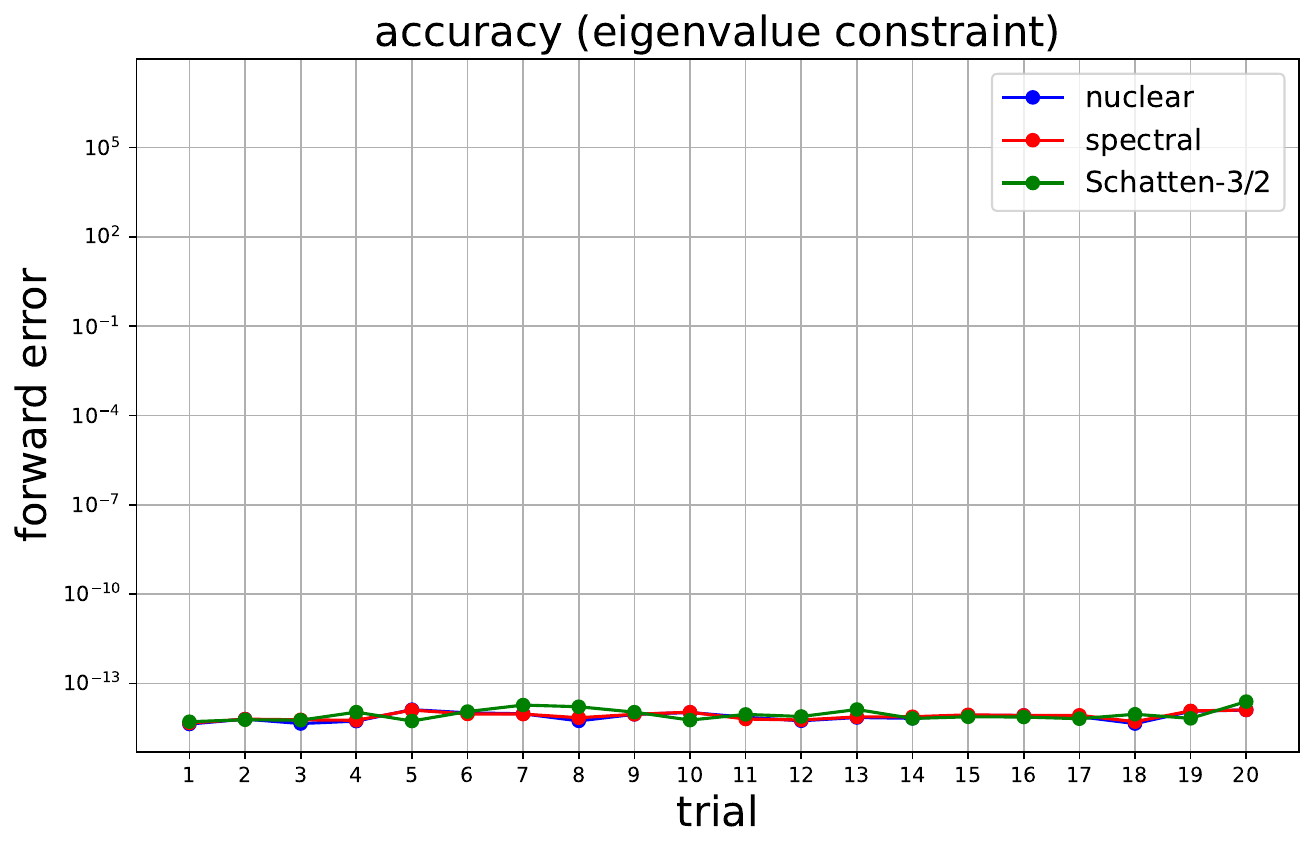}
    \endminipage\hfill
    \caption{Accuracy of Algorithm~\ref{alg:Schatten}. Again \texttt{CVX} comparison is unavailable for nonconvex or nonstandard problems.  We use the same vertical scale for easy comparison across different cases.}
    \label{fig:accuracy}
\end{figure}  
\FloatBarrier

    \subsection{Nuclear norm in system identification} 

Nuclear norm minimization arises naturally in system identification \cite{system,system2} for discrete-time linear time-invariant state-space model with inputs $u(t) \in \R^m$, outputs $x(t) \in \R$, $t = 0, \dots, 2n$. Define the input block Hankel matrix $A \in \R^{m(n+1) \times (n+1)}$ and the output Hankel matrix $X \in \R^{(n+1) \times (n+1)}$ and as follows:
\[
A = \begin{bmatrix}
    u(0) &u(1) &\cdots &u(n)\\
    u(1) &u(2) &\cdots &u(n+1)\\
    \vdots &\vdots &\ddots &\vdots\\
    u(n) &u(n+1) &\cdots &u(2n)
\end{bmatrix},
\quad 
X = \begin{bmatrix}
    x(0) &x(1) &\cdots &x(n)\\
    x(1) &x(2) &\cdots &x(n+1)\\
    \vdots &\vdots &\ddots &\vdots\\
    x(n) &x(n+1) &\cdots &x(2n)
\end{bmatrix}.
\]
Given a matrix $C \in \R^{(n+1) \times p}$ whose columns span the null space of $A$; a Hankel matrix $\widehat{X} \in \R^{(n+1) \times (n+1)}$ that represents the measured output; and a $\delta > 0$ representing noise tolerance. The system identification problem may be formulated as
\begin{equation}\label{eq:system}
\begin{tabular}{rl}
minimize  &  $\lVert XC \rVert_{\st,1}$ \\
subject to &   $\lVert \widehat{X} - X \rVert_\F \le \delta$, \\
&  $X \in \mathbb{R}^{(n+1) \times (n+1)}$ is Hankel.
\end{tabular}
\end{equation}
This formulation seeks to minimize the system order, quantified by $\lVert XC \rVert_{\st,1}$, while keeping the system consistent with measured data. Note that \eqref{eq:system} involves two different constraints: a norm constraint as well as a structure constraint. We will use this example to demonstrate the flexibility of Algorithm~\ref{alg:Schatten} in working with multiple constraints. Let
\[
n = 10, 20, \dots, 80, \quad p = \lceil 0.8n \rceil, \quad \delta = 0.5.
\]
For each $n$, we generate a full-rank $C \in \R^{(n+1) \times p}$ and a Hankel $\widehat{X} \in \R^{(n+1) \times (n+1)}$ randomly. We then solve \eqref{eq:system} with Algorithm~\ref{alg:Schatten} in the nuclear norm $\lVert \, \cdot \, \rVert_{\st,1}$ and compare our results with a solution using the \texttt{SCS} solver in \texttt{CVX}. The speed plot on the left of Figure~\ref{fig:nuclear_norm} appears to show that Algorithm~\ref{alg:Schatten} is only slightly faster than \texttt{CVX}. But this is an illusion --- in all but the $n = 10$ case, \texttt{CVX} gives an output $\widehat{X}$ that is not even feasible, never mind optimum. We will quantify the norm constraint violation as $\max\{0, \lVert \widehat{X} - X \rVert_\F - \delta\}$ and show this plot on the right of Figure~\ref{fig:nuclear_norm}. Evidently, \texttt{CVX} is unable to achieve feasibility at nearly every value of $n$, with worsening constraint violation as $n$ increases. On the other hand, Algorithm~\ref{alg:Schatten}, by its design, remains perfectly feasible.

\begin{figure}[htb]
    \minipage{0.49\textwidth}
        \includegraphics[trim={.6em 1.6ex .6em 1.2ex}, clip, width=\textwidth]{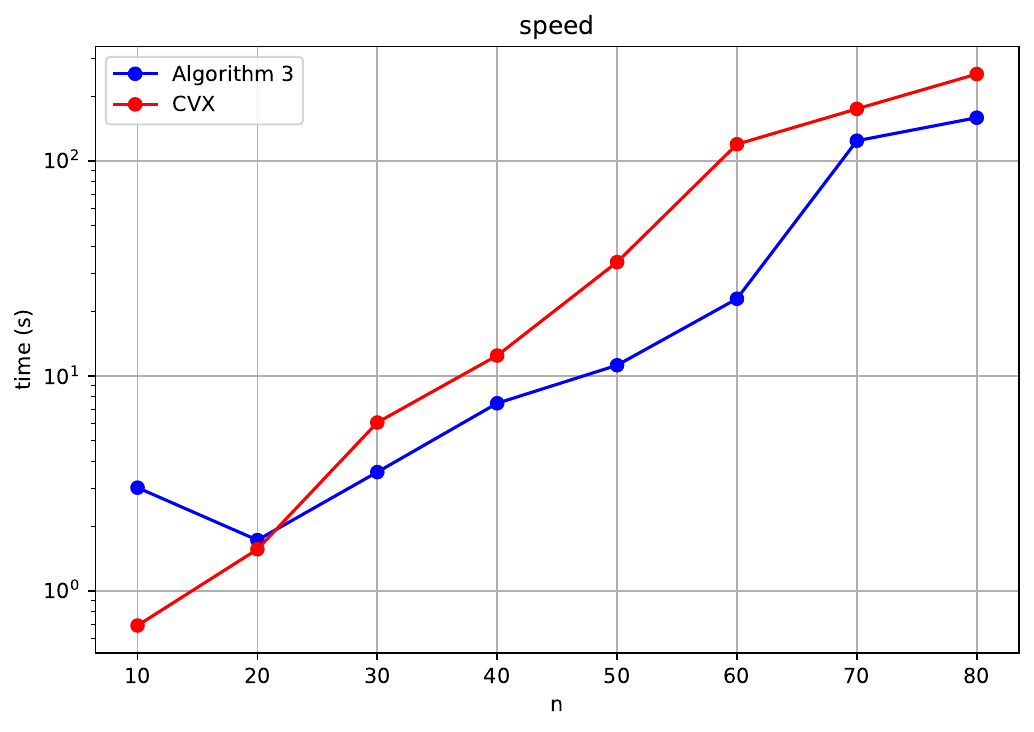}
    \endminipage\hfill
    \minipage{0.49\textwidth}
        \includegraphics[trim={.6em 1.7ex .6em 1.4ex}, clip, width=\textwidth]{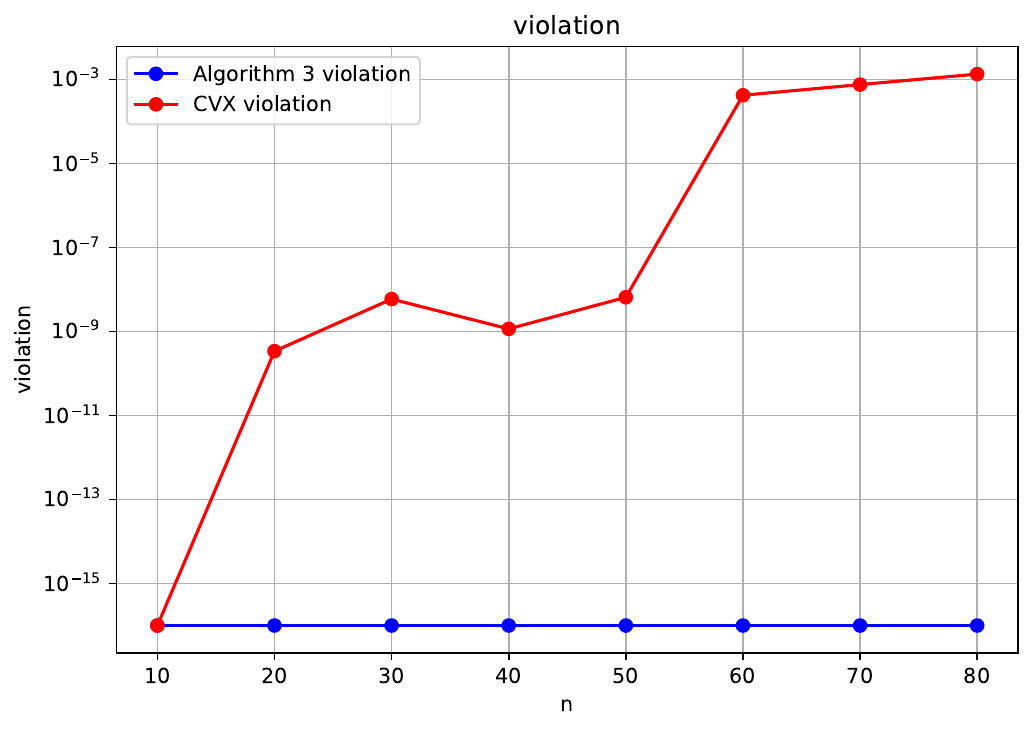}
    \endminipage\hfill
    \caption{Solving system identification problems with Algorithm~\ref{alg:Schatten} and with \texttt{CVX}. Right figure shows that \texttt{CVX} does not even converge to a feasible point.}
    \label{fig:nuclear_norm}
\end{figure}  
\FloatBarrier

\subsection{Spectral norm in CFAR detection}

In target detection problems, spectral norm naturally arises as a measure between the covariance matrix of cell under test (CUT) and the estimation matrix \cite{CFAR}. In scenarios when there is a constraint on the estimation matrix or when dimension-reduction techniques like principal component analysis have been applied, the estimation matrix must necessarily belong to $\{BXB^\tp : X \in \mathbb{S}^p_+\}$ for some $B \in \mathbb{R}^{n \times p}$.  We have the generalized matrix nearness problem
\begin{equation}\label{eq:CFAR}
    \min_{X \in \mathbb{S}^p_+} \; \lVert A - BXB^\tp \rVert_{\st,\infty},
\end{equation}
which captures dissimilarity between the covariance matrix $A \in \mathbb{S}^n_+$ in the cell under test and the mean covariance matrix of the reference cells $BXB^\tp \in \mathbb{S}^n_+$.

Recall that the positive semidefiniteness constraint is the one exceptional case in Theorem~\ref{thm:conv}, i.e., we do not have convergence guarantee for Algorithm~\ref{alg:Schatten} in this case. Indeed, a solution to \eqref{eq:CFAR} may fail to exist, as we discussed before Proposition~\ref{prop:attain}. Nevertheless, numerical evidence below shows that when a solution exists, Algorithm~\ref{alg:Schatten} will find it quickly.

Instead of taking $n=8$ as in \cite{CFAR}, we let $n = 10$, $15$, $20$, $25$ to provide a greater challenge for the two solvers. For each $n$, we generate an $A \in \mathbb{S}^{n}_+$ and a full-rank $B \in \mathbb{R}^{n \times p}$, $p = 1, \dots, n$. We then solve \eqref{eq:CFAR} with Algorithm~\ref{alg:Schatten}  and  with the \texttt{SCS} solver in \texttt{CVX}. Figure~\ref{fig:DP} shows that Algorithm~\ref{alg:Schatten} is consistently an order of magnitude faster than \texttt{CVX}.
\begin{figure}[htb]
    \centering
    \includegraphics[trim={.7em 1.6ex .7em 1ex}, clip, width=\textwidth]{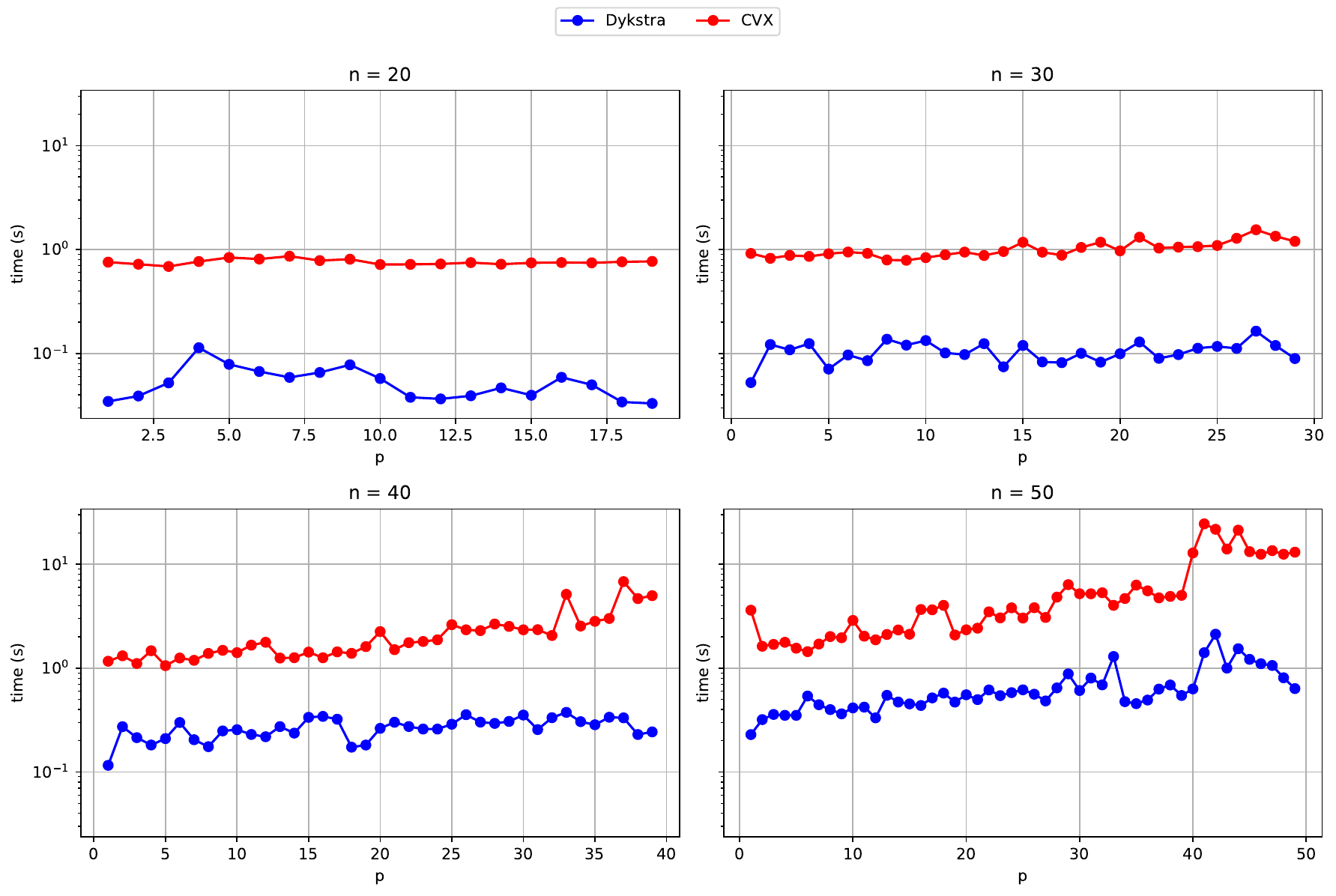}
    \caption{Solving target detection problems with Algorithm~\ref{alg:Schatten} and with \texttt{CVX}. Figures show that Algorithm~\ref{alg:Schatten} is an order of magnitude faster than \texttt{CVX}}
    \label{fig:DP}
\end{figure}  
\FloatBarrier

\section{Conclusion}

We expanded the study of generalized matrix nearness problems, deriving exact closed-form solutions in four new cases: affine objective, separable Kronecker product constraint, Kronecker rank constraint, and prescribed partial trace constraint. We showed that there is no result analogous to Mirsky Theorem for such problems. In part to address this, we proposed an iterative algorithm for generalized matrix nearness problems in any Schatten norm; the algorithm converges to a global minimizer from any starting point, and never computes a gradient or a subgradient explicitly.

As the numerical experiments show, our algorithm handles objectives like the Schatten $\frac32$-norm, which are convex and yet lie outside any standard convex solver's reach. For generalized nearness problems that do fall within the domain of such solvers, our algorithm is both faster and more accurate. This reaffirms a core premise alluded to in \cite[Section~7]{li22}: Despite the popularity of general purpose convex optimization solvers, an old-school numerical linear algebra approach allows one to exploit structures far richer than mere convexity. This yields superior computational performance, and in several cases also extends to nonconvex problems.

\subsection*{Acknowledgments} RW and LH are partially supported by a Vannevar Bush Faculty Fellowship ONR N000142312863. 

\bibliographystyle{abbrv}

\end{document}